\documentclass{amsart}
\usepackage{mathrsfs}
\usepackage{amsmath}
\usepackage{amsthm}
\usepackage{amssymb}
\usepackage{color}
\usepackage{enumitem}

\newtheorem{thm}{Theorem}
\newtheorem{lemma}[thm]{Lemma}

\newtheorem{defi}[thm]{Definition}
\newtheorem{prop}[thm]{Proposition}
\newtheorem{rk}[thm]{Remark}

\newcommand{\rr}{{\mathbb{R}}}
\newcommand{\nn}{{\mathbb{N}}}
\newcommand{\rdN}{{(\rr^3)^N}}
\newcommand{\D}{{\mathbb{D}}}
\newcommand{\E}{{\mathbb{E}}}
\newcommand{\PP}{{\mathbb{P}}}
\newcommand{\QQ}{{\mathbb{Q}}}
\newcommand{\MM}{{\mathbb{M}}}
\newcommand{\Sp}{{\mathbb{S}}}
\newcommand{\intrd}{\int_{\rr^3}}
\newcommand{\cL}{\mathcal{L}}
\newcommand{\Q}{\mathbf{Q}}
\newcommand{\cP}{\mathcal{P}}
\newcommand{\cA}{\mathcal{A}}
\newcommand{\cF}{\mathcal{F}}
\newcommand{\e}{\varepsilon}
\newcommand{\vip}{\vskip.15cm}
\newcommand{\indiq}{{\bf 1}}
\newcommand{\dd}{{\rm d}}
\newcommand{\be}{\mathbf{e}}
\newcommand{\bV}{\mathbf{V}}
\newcommand{\bZ}{{\mathbf{Z}}}
\newcommand{\bi}{\mathbf{i}}
\newcommand{\bj}{\mathbf{j}}

\newcommand{\sm}{{s-}}
\newcommand{\um}{{u-}}
\newcommand{\wto}{\rightharpoonup}

\newcommand{\beqn}{\begin{equation}}
\newcommand{\eeqn}{\end{equation}}
\newcommand{\bear}{\begin{eqnarray}}
\newcommand{\eear}{\end{eqnarray}}
\newcommand{\bean}{\begin{eqnarray*}}
\newcommand{\eean}{\end{eqnarray*}}
\newcommand{\bal}{\begin{aligned}}
\newcommand{\eal}{\end{aligned}}

\numberwithin{equation}{section}
\numberwithin{thm}{section}
\begin{document}

\title[Propagation of chaos for moderately soft potentials]{Propagation of chaos for the homogeneous 
Boltzmann equation with moderately soft potentials}

\author{Nicolas Fournier and Stéphane Mischler}
\address{N. Fournier: Sorbonne Universit\'e, LPSM-UMR 8001, Case courrier 158, 75252 Paris Cedex 05, France.}
\address{S. Mischler: CEREMADE, CNRS UMR 7534,
Universités PSL et Paris-Dauphine, Place de Lattre de Tassigny, 75775 Paris 16, France et Institut
Universitaire de France (IUF)}
\email{nicolas.fournier@sorbonne-universite.fr}
\email{mischler@ceremade.dauphine.fr}


\begin{abstract}
We show that the Kac particle system converges, as the number of particles  tends to infinity,
to the solution of the homogeneous Boltzmann equation, in the regime of moderately soft potentials,
$\gamma \in (-2,0)$ with the common notation.  This proves the propagation of chaos. We adapt the recent
work~\cite{imbertSV2024} of Imbert, Silvestre and Villani, to show that the Fisher information is 
nonincreasing in time along solutions to the Kac master equation. This estimate allows us to control 
the singularity of the interaction.
\end{abstract}

\maketitle

\section{Introduction and main results}

\subsection{The homogeneous Boltzmann equation}
The homogeneous Boltzmann equation  describes the evolution of the density $f=f(t,v)=f_t(v) \ge 0$ 
of particles with velocity $v\in \rr^3$ at time 
$t\geq 0$ in a spatially homogeneous dilute gas. It writes
\begin{equation}\label{be}
 \partial_t f = Q(f,f)    \;\;\; \text{in}\;\;\; (0,\infty) \times \rr^3 ,\quad
f(0,\cdot) = f_0   \;\;\; \text{in}\;\;\; \rr^3, 
\end{equation}
where the collision operator is given by 
\begin{eqnarray} \label{beQ}
Q(f,f) (v) :=  \intrd \dd v_* \int_{\Sp_{2}} \dd\sigma B(v-v_*,\sigma)
\big[f(v')f(v'_*) -f(v)f(v_*)\big],
\end{eqnarray}
and where in the above formula the pre-collisional velocities are given by
\begin{equation}\label{vprimeetc}
v'=v'(v,v_*,\sigma)=\frac{v+v_*}{2} + \frac{|v-v_*|}{2}\sigma, \quad 
v'_*=v'_*(v,v_*,\sigma)=\frac{v+v_*}{2} -\frac{|v-v_*|}{2}\sigma.
\end{equation}
The cross section takes the form
\begin{equation}\label{hyp0}
B(z,\sigma)\, \sin \theta = \alpha(|z|) \beta(\theta)
\end{equation}
for some $\alpha :\rr_+\to \rr_+$ and some $\beta : (0,\pi]\to \rr_+$,
where $\theta=\theta(z,\sigma)$ is defined by 
$\cos \theta = \frac{z}{|z|} \cdot \sigma$.

\vip
We will always assume that this cross section corresponds to the situation where particles repel 
each other according to a force proportional to $1/r^s$, 
with $s\in (2,\infty)$.  In that case, we have
\begin{equation}\label{hyp1}
\alpha(r)=r^\gamma \;\;\text{and}\;\; c_1 \theta^{-1-\nu}\leq \beta(\theta)\leq c_2 \theta^{-1-\nu},
\;\; \hbox{with}\;\; \gamma=\frac{s-5}{s-1}\in (-3,1) \;\;\text{and}\;\; \nu=\frac{2}{s-1}
\end{equation} 
for some constants $c_2>c_1>0$. 
One speaks of hard spheres when $\gamma=1$, hard potentials when $\gamma\in (0,1)$, 
Maxwell molecules when $\gamma=0$ and soft potentials when $\gamma\in (-3,0)$.
The present paper is concerned with moderately soft potentials, i.e. $\gamma \in (-2,0)$. 

\vip
The main physical features of the Boltzmann equation is that any solution conserves mass, 
momentum and kinetic energy and that its entropy is decaying. 
From Galilean invariance, we may (and will) always normalize the initial datum, and we 
thus have
\beqn\label{eq:normalizationf}
\int_{\rr^3} f_t(v) \dd v=1, \quad \int_{\rr^3} f_t(v) v  \dd v=0,  \quad \int_{\rr^3} f_t(v) |v|^2  \dd v=3, 
\quad \forall \, t \ge 0, 
\eeqn
as well as the entropy bound
$$
H(f_t) := \int_{\rr^3} f_t(v) \log f_t(v)  \dd v \le H(f_0), \quad \forall \, t \ge 0. 
$$

In a recent paper~\cite{guillenSilvestre202}, Guillen and Silvestre have established
that the Fisher information 
\beqn\label{eq:Fisher1}
I_1(f) := \int_{\rr^3} \frac{|\nabla f(v)|^2}{f(v)} \dd v
\eeqn
is nonincreasing in time along the solutions of the Landau equation. This was extended by Imbert, Silvestre 
and Villani in~\cite{imbertSV2024} to the Boltzmann equation~\eqref{be}, for any 
physical kernels~\eqref{hyp0}-\eqref{hyp1}. 
Although such a property result was known for Maxwell molecules, 
see McKean~\cite{MR214112}, Toscani~\cite{MR1164495} and Villani~\cite{MR1646804}, 
it was rather unexpected for other kernels.
This remarkable result and the possibility of extending it 
to the Kac particle system are the starting point and a cornerstone argument in the present work. 
 
\subsection{The Kac particle systems}
\label{subsec:KacSystem}
Kac~\cite{MR84985} introduced a stochastic particle system, expected to approximate the Boltzmann equation.
This system is composed of a large number $N$ of particles characterized by their 
velocities $v_1,\dots,v_N$ and such that 
each couple of particles with velocities $(v_i,v_j)$ randomly collides, for each $\sigma\in\Sp_2$, 
at rate $\frac{B(v_i-v_j,\sigma)}{2(N-1)}$ and are then replaced by particles with velocities 
$v'(v_i,v_j,\sigma)$ and $v'_*(v_i,v_j,\sigma)$.
Equivalently, the Kac particle system is a $\rdN$-valued Markov process $(\bV^N_t)_{t\geq 0}$
with generator $\cL_N$ defined (at least informally) by
\begin{equation*}
\cL_N \phi(v)= \frac 1 {2(N-1)} \sum_{1\leq i \neq  j \leq N} \int_{\Sp_2} 
[\phi(v'_{ij}) - \phi(v)]
B(v_i-v_j,\sigma)\dd\sigma,
\end{equation*}
for any $v=(v_1,\dots,v_N) \in \rdN$ and any $\phi:\rdN\mapsto \rr$ sufficiently regular.
Here, we have set
\begin{equation}\label{vpij}
v'_{ij}=v'_{ij}(v,\sigma)= v + (v'(v_i,v_j,\sigma)-v_i)\be_i +  (v'_*(v_i,v_j,\sigma)-v_j)\be_j \in \rdN,
\end{equation}
where for $h \in \rr^3$, $h\be_i=(0,\dots,0,h,0,\dots,0)$ with $h$ at the $i$-th place.

\vip
It turns out that the law  $F^N_t$ of $\bV^N_t$ satisfies the $N$-particle linear Boltzmann 
equation, also called Kac master equation, 
\beqn\label{eq:BlNeq}
\partial_t F^N_t= \cL_N F^N_t   \;\;\; \text{in}\;\;\; (0,\infty) \times (\rr^3)^N ,\quad
F^N(0,\cdot) = F^N_0   \;\;\; \text{in}\;\;\; (\rr^3)^N.
\eeqn
That equation has the same main physical properties as the Boltzmann equation since its
solutions also conserve mass, momentum and kinetic energy and have a decaying entropy. 
As alluded to above, we will establish that the $N$-particle Fisher information 
\beqn\label{eq:FisherN}
I_N(F)=\frac 1 N \int_{\rdN} \frac{|\nabla F(v)|^2}{F(v)} \dd v
\eeqn
is nonincreasing along solutions of~\eqref{eq:BlNeq}.

\subsection{Propagation of chaos} \label{subsec:propgation-chaos}
We recall some more or less classical notions of chaos which have been introduced 
in Kac's paper~\cite{MR84985} and later in Sznitman~\cite{MR1108185}, Carlen et al.~\cite{MR2580955}
and in~\cite{MR3188710}, 
to which we refer for more details.

\vip
For a Polish space $E$, we denote by $\cP(E)$ the set of probability measures on $E$. 
We say that  $G^N \in \cP(E^N)$ is symmetric (and write $G^N \in \cP_{\! s}(E^N)$) if 
$\langle G^N,\varphi_\sigma\rangle = \langle G^N,\varphi\rangle$ for any $\varphi \in C_b(E^N)$ and any 
permutation $\sigma$ of $\{1,\dots,N\}$, where 
$\varphi_\sigma(Z) := \varphi(Z_{\sigma(1)},\dots,Z_{\sigma(N)})$  for any $Z \in E^N$. 
For $1 \le \ell \le N$, we define the $\ell$-th marginal $G^N_\ell \in \cP(E^\ell)$
 by $\langle G^N_\ell, \varphi \rangle := \langle G^N, \varphi \otimes {\bf 1}^{N-\ell} \rangle$, 
for any $\varphi \in C_b(E^\ell)$. 
For $g \in \cP(E)$, we say that a sequence  $(G^N)_{N\geq 1}$ of $\cP(E^N)$
is (weakly) $g$-chaotic if 
$$
G^N_\ell \wto g^{\otimes \ell} \ \hbox{ weakly in } \ \cP(E^\ell) \ \hbox{ as } \ N \to \infty,
$$
for any fixed $\ell \ge 1$ (or equivalently for at least one fixed $\ell \ge 2$). 

\vip

When $E = \rr^3$,  we define the $N$-particle entropy of $G^N \in \cP_{\!s}((\rr^3)^N)$ such that 
$m_2(G^N) := \langle G^N,|v_1|^2 \rangle< \infty$ by
$$ 
H_N(G^N) := \frac1N  \int_{\rr^{3N}} G^N(v) \log G^N(v) \dd v.
$$
Consider now a sequence of symmetric probability measures $(G^N)_{N\geq 1}$ with $\sup_{N\geq 1}m_2(G^N)<\infty$.
We say that  $(G^N)_{N\geq 1}$ is entropically $g$-chaotic if  
$$
G^N_1 \wto g \hbox{ weakly in } \cP(\rr^3) \hbox{ and } H_N(G_N) \to H(g). 
$$

A $E^N$-valued random vector $\bZ^N$ is said to be exchangeable if its law is symmetric. 
For any $Z = (z_1, ...,z_N) \in E^N$, we define the empirical measure 
\beqn\label{def:EmpirMeasure}
\mu^N_Z := \frac1N \sum_{i=1}^N \delta_{z_i} \in \cP(E).
\eeqn
Consider a family $(\bZ^N)_{N\geq 1}$ of exchangeable random vectors, with  $\bZ^N$ valued in $E^N$.
For $g \in \cP(E)$,  we say that $(\bZ^N)_{N\geq 1}$ is $g$-chaotic if its sequence of laws is $g$-chaotic, 
or equivalently, see~\cite[Proposition~2.2]{MR1108185},
if the associated  $\cP(E)$-valued random variable  $\mu^N_{\bZ^N} $ converges 
to the deterministic random variable $g$:
\beqn\label{defKacProba} 
\mu^N_{\bZ^N}  \, \Rightarrow \, g \quad \hbox{in probability as} \quad N\to\infty. 
\eeqn
Here $\Rightarrow$ refers to the weak convergence of measures (tested against $C_b(E)$-functions).

\vip

Similarly, when $E=\rr^3$, we say that $(\bZ^N)_{N\geq 1}$ is 
entropically $g$-chaotic if the associated sequence of laws 
is entropically $g$-chaotic.

\subsection{Main result}

We are now in position to formulate the main results of the paper. 

\begin{thm}\label{th:chaos}
Consider a kernel $B$ satisfying~\eqref{hyp0}-\eqref{hyp1} with $\gamma \in (-2,0)$ 
and a nonnegative initial condition $f_0$ satisfying the normalization 
$$
\int_{\rr^3} f_0(v) \dd v=1, \quad \int_{\rr^3} f_0(v) v  \dd v=0,  \quad \int_{\rr^3} f_0(v) |v|^2  \dd v=3, 
$$
  and with finite Fisher information. Denote by $f \in C([0,\infty);L^1(\rr^3))$ the unique weak solution
to the Boltzmann equation~\eqref{be} with nonincreasing Fisher information and which 
which satisfies~\eqref{eq:normalizationf}.
For any $N\geq 2$, there exists
a $\rdN$-valued Kac particle system $(\bV^N_t)_{t\geq 0}$ with initial law $f_0^{\otimes N}$ 
and for all $t\geq 0$, 
$(\bV^N_t)_{N\geq 2}$ is entropically $f_t$-chaotic.
\end{thm}

Precise definitions of the objects involved in this theorem as well as more accurate or more general 
statements will be given in the next sections. In particular, we prove propagation of chaos in the
sense of trajectories, see Theorem~\ref{mr2}.
The main novelty is that we are able to deal with moderately soft potentials.

\vip

Some propagation of chaos results of Kac's system to the Boltzmann equation have been
obtained for Maxwell molecules by Kac~\cite{MR84985}, Graham-Méléard~\cite{MR1428502},
Desvillettes-Graham-Méléard~\cite{MR1720101}, see also~\cite{MR3069113} and 
Cortez-Fontbona~\cite{MR3769742}. Hard spheres (and hard potentials with angular cutoff) were
treated by Grünbaum~\cite{MR334788}, Sznitman~\cite{MR753814} and Norris~\cite{norris}.
Heydecker~\cite{MR4419606} was able to deal with hard potentials $\gamma \in (0,1)$ without cutoff.
To our knowledge, the case of (singular) soft potentials was left completely open.

\vip
It is worth mentioning a slightly different particle system, the so-called Nanbu approximation. 
This system is somehow easier to manage than Kac's system, but it has
less pertinent physical meaning, since the collisions are not symmetric.
For this system, propagation of chaos has been established in 
Nanbu~\cite{MR722216}, Graham-Méléard~\cite{MR1428502}, Xu~\cite{MR3784497}, see also~\cite{MR3456347}. 
Together, these works cover the range $\gamma \in (-1,1]$.
\vip
 We also refer to the works of Fontbona, Guérin and Méléard~\cite{MR2475665}, Carrapatoso~\cite{MR3422644},
see also~\cite{MR3304746,MR3572320,MR3621429} and the references therein for some propagation of chaos
of the Kac system (when $\gamma \in [0,1]$) and of the Nanbu system (when $\gamma \in (-2,0]$) 
to the Landau equation. Let also mention that the convergence of subsequences to the hierarchy, 
for the very singular Coulomb potential 
$\gamma=-3$, has been established by Miot, Pulvirenti and Saffirio~\cite{MR2765750}, 
and more recently by Carrillo and Guo~\cite{carrillo2025FisherLandauSystem}. Finally, propagation of chaos for
Kac's system to the Landau equation 
has been established by Tabary~\cite{tab} and Feng-Wang~\cite{fw} 
for very soft potentials and the Coulomb case ($\gamma \in [-3,-2)$). Let us also mention the work of 
Du~\cite{kaidu}, who derives the Landau equation (with $\gamma \in [-3,1]$) from the Kac particle system 
associated to a Boltzmann equation with a specific collison kernel in the asymptotics of grazing collisions.

\vip 
Our proof follows a general strategy introduced by Sznitman~\cite{MR753814}, revisited for singular interactions 
in~\cite{MR3254330} in the context of the vortex system. 

\vip
The first key argument is to establish the 
decay of the Fisher information along the solutions to~\eqref{eq:BlNeq}, 
following the ideas of the recent papers dedicated to the Landau and Boltzmann equations 
of Guillen and Silvestre~\cite{guillenSilvestre202},
of Imbert, Silvestre and Villani~\cite{imbertSV2024} and of Villani~\cite{villani2025}.
Let us also mention that the works~\cite{carrillo2025FisherLandauSystem,tab} cited above
rely on the very same Fisher information decay, derived for the Kac particle system corresponding to
the Landau-Coulomb equation.
\vip
We thus get a uniform bound on the $N$-particle Fisher information. Next, 
following~\cite{MR3254330},  this uniform bound
makes possible to pass to the limit as $N \to \infty$ and to prove that $(\bV^N_t)_{t\geq 0}$ converges in law
to a nonlinear stochastic process $(\bV_t)_{t\geq 0}$, of  which the law is supported, roughly speaking,  
on solutions to the Boltzmann equation. It turns out that, using some level-$3$ Fisher information 
results picked up from~\cite{MR3188710}, the Fisher information of these solutions 
are integrable in time, so that we may use the uniqueness result established in~\cite{MR2398952} 
in order to conclude.
\vip

\begin{rk}
For very soft potentials $\gamma \in (-3,2]$, 
we can prove, with almost the same arguments, points (i) and (ii) of Theorem~\ref{mr2}
below, i.e. roughly speaking, some compactness result for the particle system and that any (possibly random) 
limit point a.s. solves~\eqref{be}. However, we have not sufficiently regularity information on this limit point
to apply the uniqueness result of~\cite{MR2398952}.
Since such a result is rather weak and since this would complicate the presentation, 
in particular due to the assumptions of~\cite{imbertSV2024} on the cross section, which cover
all physical cases but are complicated to write down when $\gamma \in (-3,-2]$, 
we decided to restrict our study to
the case $\gamma \in (-2,0)$.
\end{rk}

\subsection{Strategy of the proof}

As explained above, our proof makes use of
(a) the results of Imbert-Villani-Silvestre~\cite{imbertSV2024} about the decay of the Fisher information for
the Boltzmann equation (that we extend to the Kac particle system), 
(b) the tools developed in~\cite{MR3254330} and~\cite{MR3188710}  to control
the singularity of the interaction by the Fisher information (in the context of continuous processes), 
(c) some ideas taken in Sznitman~\cite{MR1108185,MR753814} and
Méléard~\cite{MR1431299} about the propagation of chaos for càdlàg processes (with non-singular interactions), 
and (d) the uniqueness result of~\cite{MR2398952}, which implies that there is at most one weak solution to the 
Boltzmann equation with time-integrable Fisher information when $\gamma \in (-2,0)$.

\vip

As we shall see, something quite disappointing is that the bound 
$\sup_{N\geq 2} \sup_{t\geq 2} I_N(F^N_t)<\infty$
shows that for $(\mu_t)_{t\geq 0}$ a (random) limit point of the empirical measure of the Kac
particle system, we have $\sup_{t\geq 0} \E[I_1(\mu_t)]<\infty$. This of course implies that
$t\mapsto I_1(\mu_t) \in L^1_{loc}([0,\infty))$ a.s. 
but not that $t\mapsto I_1(\mu_t) \in L^\infty_{loc}([0,\infty))$ a.s.
Such a bound would imply that $(\mu_t)_{t\geq 0} \in L^\infty(\rr_+,L^{3}(\rr^3))$. Although the situation is much less
clear for the Boltzmann equation, such an estimate would be sufficient to have uniqueness in the case of 
the Landau equation in the Coulomb case, see Chern-Gualdani~\cite{MR4557015} and 
Golding-Gualdani-Loher~\cite{MR4904566}. 
\vip

Tabary~\cite{tab} manages to treat the Landau equation when $\gamma \in [-3,0)$.
He is able to take advantage of the Fisher 
information dissipation to show that
any limit point $(\mu_t)_{t\geq 0}$ of the empirical measure actually lies in $L^1_{loc}(\rr_+,L^\infty(\rr^3))$,
which allows him to apply the uniqueness result of~\cite{funic}.
Such a program might be feasible for the Boltzmann equation for very soft potentials $\gamma \in (-3,-2]$,
with much more involved  computations.

\subsection{Plan of the paper} Section~\ref{sec:Fisher} is concerned with the Fisher information for the 
$N$-particle linear Boltzmann equation: we establish its decay in time and deduce some useful consequences. 
In Section~\ref{sec:Boltzmann}, we recall some results about the existence and uniqueness for the Boltzmann 
equation. Section~\ref{sec:Kac} is devoted to the Kac $N$-particle system: we state an existence result 
and describe its main properties.
In Section~\ref{sec:NonlinearStochastic}, we introduce a nonlinear stochastic process associated 
with the Boltzmann
equation, as initiated by Tanaka~\cite{MR512334} for Maxwell molecules.
In Section~\ref{sec:KACtoNL}, we prove that this  nonlinear 
stochastic process is the limit of the $N$-particle system.
Finally, we conclude the proof of Theorem~\ref{th:chaos} in  Section~\ref{sec:propagation}.
 
\section{The Fisher information}\label{sec:Fisher}

The goal of this section is mainly to show the time-decay of the Fisher information for the $N$-particle system.
This relies on the following result, found in the article of Imbert, Silvestre and Villani~\cite{imbertSV2024}.

\begin{lemma}\label{isv}
Assume~\eqref{hyp0}, with $\int_0^\pi \theta^2\beta(\theta)\dd \theta <\infty$ and with
$\alpha(r)=(\e^2+r^2)^{\gamma/2}$ for some $\e \geq 0$ 
and some $\gamma \in (-2,0)$. 
For $F:(\rr^3)^2 \to \rr_+$ and $v_1,v_2\in \rr^3$, set
$$
\Q(F)(v_1,v_2)=\int_{\Sp_2} [F(v'(v_1,v_2,\sigma),v'_*(v_1,v_2,\sigma))-F(v_1,v_2)] B(v_1-v_2,\sigma)\dd \sigma.
$$ 
For any function $F:(\rr^3)^2 \to \rr_+$ which is symmetric, that is $F(v_1,v_2)=F(v_2,v_1)$, it holds
\begin{equation}\label{isvin}
\int_{(\rr^3)^2} \Big(\frac{2\nabla_1 F \cdot \nabla_1 \Q(F)}{F}-\frac{|\nabla_1 F|^2 \Q(F)}{F^2} \Big) 
\dd v_1\dd v_2 \leq 0,
\end{equation}
where $\nabla_1 F$ is the gradient of $F$ with respect to its first variable $v_1\in \rr^3$.
\end{lemma}

\begin{proof}
Let us check that the assumptions of~\cite[Theorem~1.2]{imbertSV2024} are satisfied. 
First, $b(\cos\theta)=\beta(\theta)$ satisfies~\cite[(1.2)]{imbertSV2024}, 
i.e. $\int_0^\pi (1-\cos\theta) \beta(\theta)\dd \theta<\infty$.
Next, we need to verify that $\frac{r|\alpha'(r)|}{\alpha(r)}\leq 
2\sqrt{\Lambda_b}$, with $\Lambda_b$ defined in~\cite[(1.4)]{imbertSV2024}. 
But, as seen in~\cite[Proposition~7.11]{imbertSV2024}, it holds that $\Lambda_b\geq 1$. Since
$$
\frac{r|\alpha'(r)|}{\alpha(r)}=\frac{r^2 |\gamma| (\e^2+r^2)^{\gamma/2-1}}
{(\e^2+r^2)^{\gamma/2}}= \frac{r^2 |\gamma|}{ \e^2+r^2}\leq |\gamma| \leq 2 \leq 2 \sqrt{\Lambda_b},
$$
all the assumptions of~\cite[Theorem~1.2]{imbertSV2024} are satisfied.
\vip
By~\cite[(3.3)]{imbertSV2024}, which is shown during the proof of~\cite[Theorem~1.2]{imbertSV2024},
that for any symmetric  function $F:(\rr^3)^2 \to \rr_+$, it holds 
$$
\langle I_2'(F), \Q(F)\rangle \leq 0,
$$
where $I_2'(F)$ is the Gâteau derivative of $I_2$, see \eqref{eq:FisherN}. This precisely rewrites as
$$
\frac12\int_{(\rr^3)^2} \Big(\frac{2\nabla F \cdot \nabla \Q(F)}{F}-\frac{|\nabla F|^2 \Q(F)}{F^2} \Big) 
\dd v_1\dd v_2 \leq 0.
$$
The LHS of this inequality equals the LHS of~\eqref{isvin} by symmetry of $F$.
\end{proof}

We can now state the main result of this section, generalizing Carrillo and Guo~\cite{carrillo2025FisherLandauSystem}
which proves a similar result on the Kac's system for the Landau-Coulomb equation.

\begin{prop}\label{prop:decay}
Assume~\eqref{hyp0} with $\int_0^\pi \theta^2\beta(\theta)\dd \theta <\infty$ and 
$\alpha(r)=(\e^2+r^2)^{\gamma/2}$ for some $\e \geq 0$ 
and some $\gamma \in (-2,0)$. 
For any $N\geq 2$, any nonnegative symmetric smooth 
enough solution $F^N$ to the $N$-particle linear Boltzmann equation~\eqref{eq:BlNeq}, 
$t \mapsto I_N(F^N_t)$ is nonincreasing.
\end{prop}

\begin{proof}
 A simple computation, recalling the definition of $I_N$
in~\eqref{eq:FisherN}, shows that
$$
\frac{\dd}{\dd t} I_N(F_t^N) = \int_{(\rr^3)^N} \Big(2\frac{\nabla F_t^N(v)\cdot \nabla \partial_t F_t^N(v) }{F_t^N(v)}
-\frac{|\nabla F_t^N(v)|^2 \partial_t F_t^N(v)}{(F_t^N(v))^2}\Big) \dd v.
$$
We claim that
\begin{equation}\label{prop:decay-eq1}
\Delta_N(F):= \int_{(\rr^3)^N} \Big( 2\frac{\nabla F\cdot \nabla \cL_N F}{F} 
- \frac{|\nabla F|^2 \cL_N F}{F^2}\Big)\dd v \leq 0,
\end{equation}
for any (smooth) probability density $F$ on $(\rr^3)^N$,
from what the announced decay follows.  
\vip

From the very definition of $\cL_N$ 
in Section~\ref{subsec:KacSystem}, we may write 
$$
\cL_N = \frac1{2(N-1)} \sum_{1\leq i \neq  j \leq N} \cL_N^{ij}, \quad \text{where} \quad  \cL_N^{ij} F 
:= \int_{\Sp_2} (F'_{ij} - F) B_{ij} \dd\sigma, 
$$
with $F'_{ij} := F(v'_{ij})$, $F = F(v)$, $B_{ij} := B(v_i-v_j, \sigma)$. Thus
$$
\Delta_N(F)=\sum_{\ell=1}^N \int_{(\rr^3)^N} \Big(2\frac{\nabla_\ell F\cdot \nabla_\ell \cL_N F}{F}
- \frac{|\nabla_\ell F|^2 \cL_N F}{F^2}  \Big)\dd v=
\frac1{2(N-1)}\sum_{\ell=1}^N \sum_{1\leq i\neq j\leq N} \Delta_N^{\ell,i,j}(F),
$$
where $\nabla_\ell$ stands for the gradient associated to the variable $v_\ell \in \rr^3$, and where
$$
\Delta_N^{\ell,i,j}(F):=\int_{(\rr^3)^N} \Big(2\frac{\nabla_\ell F\cdot \nabla_\ell \cL_N^{ij} F}{F}
-\frac{|\nabla_\ell F|^2 \cL_N^{ij} F}{F^2}\Big) \dd v.
$$
We now show that each term $\Delta_N^{\ell,i,j}(F)$ is nonpositive.

\vip
{\it The case $\ell \notin \{i,j\}$.} We observe that
$$
\nabla_\ell \cL_N^{ij} F = \int_{\Sp^2} [\nabla_\ell F'_{ij}  - \nabla_\ell F ] B_{ij} \dd\sigma,
$$
with $\nabla_\ell F'_{ij} = (\nabla_\ell F)(v'_{ij})$ and $\nabla_\ell F  = (\nabla_\ell F)(v)$. Consequently,
\bean
\Delta_N^{\ell,i,j}(F)&=&   \int_\rdN \Big( 2\frac{\nabla_\ell F}{F} \cdot  \int_{\Sp^2} 
[\nabla_\ell F'_{ij}  - \nabla_\ell F] B_{ij} \dd\sigma    -  \frac{|\nabla_\ell F|^2}{F^2} 
\int_{\Sp^2} [F'_{ij} -   F] B_{ij} \dd\sigma \Big) \dd v
\\
&=&    \int_\rdN \!\! \int_{\Sp^2} \Big(2 \frac{\nabla_\ell F}{F} \cdot  \nabla_\ell F'_{ij}  - 
\frac{|\nabla_\ell F|^2}{F} - \frac{|\nabla_\ell F|^2}{F^2}  F'_{ij}  \Big)B_{ij} \dd\sigma \dd v.
\eean
As is well-known, see e.g.~\cite[Subsection~2.2 ]{MR1942465}, for any measurable $\varphi:(\rr^3)^2\to \rr_+$,
$$
\int_{(\rr^3)^2} \!\!\int_{\Sp^2} \varphi(v,v_*) B(v-v_*,\sigma) \dd\sigma \dd v_*\dd v
=\int_{(\rr^3)^2} \!\!\int_{\Sp^2} \varphi(v',v_*') B(v-v_*,\sigma) \dd\sigma \dd v_*\dd v.
$$
Consequently, for any measurable $\Phi:\rdN\to \rr_+$, 
\begin{equation}\label{vvvv}
\int_\rdN \!\! \int_{\Sp^2} \Phi(v) B_{ij} \dd\sigma \dd v =
\int_\rdN \!\! \int_{\Sp^2} \Phi(v_{ij}') B_{ij} \dd\sigma \dd v.
\end{equation}
We conclude that 
\bean 
\Delta_N^{\ell,i,j}(F)&=&    \int_\rdN \!\! \int_{\Sp^2} \Big(2 \frac{\nabla_\ell F}{F} \cdot  \nabla_\ell F'_{ij}  - 
\frac{|\nabla_\ell F'_{ij}|^2}{F'_{ij}} - \frac{|\nabla_\ell F|^2}{F^2}  F'_{ij}  \Big)B_{ij} \dd\sigma \dd v
\\
&=&  -  \int_\rdN \!\! \int_{\Sp^2} \Bigl| \frac{\nabla_\ell F}{F} \sqrt{F'_{ij}}   - 
\frac{\nabla_\ell F'_{ij}}{\sqrt{F'_{ij}}} \Bigr|^2  B_{ij} \dd\sigma \dd v \leq 0.
\eean

{\it The case $\ell \in \{i,j\}$.} We e.g. assume that $\ell=i=1$ and $j=2$ and write
\begin{align*}
\Delta_N^{1,1,2} (F)=&\int_{(\rr^3)^N} \Big(2\frac{\nabla_1 F\cdot \nabla_1 \cL_N^{12} F}{F}
-\frac{|\nabla_1 F|^2 \cL_N^{12} F}{F^2}\Big) \dd v\\
=& \int_{(\rr^3)^{N-2}} \Big[\int_{(\rr^3)^2}  
\Big(2\frac{\nabla_1 F_{v_{3N}} \cdot \nabla_1 \Q (F_{v_{3N}})}{F_{v_{3N}}}
-\frac{|\nabla_1 F_{v_{3N}}|^2 \Q (F_{v_{3N}})}{F_{v_{3N}}^2}\Big) \dd v_1 \dd v_2 \Big] \dd v_{3,N}, 
\end{align*}
where $v_{3N}=(v_3,\dots,v_N)$, where $\dd v_{3N}=\dd v_3\dots \dd v_N$, where 
$F_{v_{3N}} : (\rr^3)^2\to\rr_+$ is defined by $F_{v_{3N}}(v_1,v_2)=F(v_1,v_2,v_{3N})$, and where $\Q$ was defined
in Lemma~\ref{isv}. We used that 
$$\cL_N^{12}F(v_1,\dots,v_N)=\Q(F_{v_{3N}})(v_1,v_2).
$$
The integral on $(\rr^3)^2$ is nonpositive for each $v_{3N}$
by Lemma~\ref{isv}, and the proof is complete.
\end{proof}

Here are the properties of the Fisher information (see~\eqref{eq:FisherN}) we will use.

\begin{lemma}\label{pfish}
(a) If $f$ is a probability density on $\rr^3$, then for all $N\geq 1$,
$$
I_N(f^{\otimes N})=I_1(f). 
$$

(b) There is a constant $C_0\in (0,\infty)$ such that for all probability density $f$ on $\rr^3$,
$$
||f||_{L^3(\rr^3)}\leq C_0 I_1(f).
$$

(c) For all $N\geq 2$, all symmetric probability density $F^N$ on $\rdN$, all $k\in \{1,\dots,N\}$,
denoting by $F^N_k$ the $k$-marginal of $F^N$, we have
$$
I_k(F^N_k) \leq I_N(F^N).
$$

(d) For any $a\in (-2,0)$, there is a constant $C_a\in (0,\infty)$ such that for all $N\geq2$, all
symmetric probability density $F^N$ on $\rdN$ and for all $(V^1,\dots,V^N)\sim F^N(v)\dd v$, there holds
$$
\E[|V^1-V^2|^a]\leq C_a (1+  I_N(F^N)).
$$

(e) Fix $k\geq 2$ and consider a family of symmetric probability densities $(F_n)_{n\geq 1}$ on $(\rr^3)^k$ 
converging weakly to some symmetric probability measure $F$ on $(\rr^3)^k$. Then 
$$
I_k(F) \leq \liminf_n I_k(F_n).
$$

(f) Consider, for each $N\geq 2$, a symmetric probability density $F^N$ on $\rdN$, as well as
$(V^{1,N},\dots,V^{N,N})\sim F^N(v)\dd v$. Assume that $\mu_N:=\frac 1 N \sum_{i=1}^N \delta_{V^{i,N}}$
converges in law, in $\cP(\rr^3)$, to some (possibly random) probability measure $\mu$ on $\rr^3$. Then
$$
\E[I_1(\mu)] \leq \liminf_N I_N(F^N).
$$
\end{lemma}

\begin{proof}
All these properties are classical. For (a), we refer to~\cite[Lemma~3.2]{MR3188710}. 
\vip
Property (b) is nothing but the Sobolev embedding $H^1(\rr^3) \subset L^6(\rr^3)$ applied to
$\sqrt{f}$ (observe that $I_1(f)=4\int_{\rr^3} |\nabla \sqrt f (v)|^2 \dd v$).
\vip
Point (c) is proved in Carlen~\cite[Theorem~3]{MR1132315}, see also~\cite[Lemma~3.7]{MR3188710}.
\vip
For (d), we argue as in~\cite[Lemma~3.3]{MR3254330}.
We introduce the unitary linear transformation
$$
\forall \, (x_1,x_2) \in \rr^2 \quad \Phi  (x_1,x_2) =  \frac1{\sqrt 2}  
\bigl(x_1- x_2,x_1+x_2\bigr) =:   (y_1,y_2).
$$
For $\tilde F^N_2 := F^N_2 \circ \Phi^{-1}$, which is 
the law of $\frac1{\sqrt 2}  \bigl(V_1- V_2,V_1+V_2\bigr)$ and for
$\tilde f$ the first marginal of $\tilde F^N_2$,
$$
I_1(\tilde f) \le  2I_2(\tilde F^N_2)=  2 I_2(F^N_2) \le 2 I_N(F^N).
$$
Indeed, the first inequality is obvious, the equality follows from a simple substitution, and the last 
inequality follows from (c).
On the other hand, we have
\bean
\E[|V^1-V^2|^a]  &=& 
\int_{\rr^3 \times \rr^3} {F^N_2(x_1,x_2)} |x_1-x_2|^a \, \dd x_1\dd x_2
\\
&=&  2^{a/2} \int_{\rr^3 \times \rr^3} \tilde F^N_2(y_1,y_2) {|y_1|^a} \, \dd y_1\dd y_2
\\
&=& 2^{a/2}  \int_{\rr^3} \tilde f (y) |y|^a \, \dd y \\
&\le&   2^{a/2} +  2^{a/2} \int_{|y| \le 1} \tilde f (y) |y|^a \, \dd y.
\eean
For the last term, we have 
$$
\int_{|y| \le 1} \tilde f (y) |y|^a \, \dd y \le  \| \tilde f \|_{L^3(\rr^3)} \Bigl(  \int_{|y| \le 1}  
|y|^{3a/2} \, \dd y \Bigr)^{2/3} =C  \| \tilde f \|_{L^3(\rr^3)},
$$
thanks to  the H\"older inequality and the condition $a > -2$ (whence $3a/2>-3$).
The conclusion follows, since $\| \tilde f \|_{L^3(\rr^3)} \leq C_0 I_1(\tilde f)$ by (b) and
since we have seen that $I_1(\tilde f) \le  2 I_N(F^N)$.

\vip
Property (e) is nothing but~\cite[Lemma~3.5]{MR3188710}. 

\vip
For (f), let us denote by  $\pi \in \cP(\cP(\rr^3))$ the weak limit of $\cL(\mu^N_V)$ and set 
$\pi_k=\int_{\cP(\rr^3)}f^{\otimes k}\pi(\dd f)$ (which belongs to $\cP((\rr^3)^k$) 
for all $k\geq 1$. By~\cite[Theorem~5.3.-(1)]{MR3188710}, it holds that
$F^N_k \to \pi_k$ weakly as $N\to \infty$ for all $k\geq 1$, where $F^N_k$ is the $k$-marginal of $F^N$.
We thus deduce from~\cite[Theorem~5.7.-(2)]{MR3188710} that
$$
\E[I_1(\mu)] = \int_{\cP(\rr^3)} I_1(f) \pi (\dd f) \leq \liminf_N I_N(F^N)
$$
as desired.
\end{proof}

\section{Reminder about the Boltzmann equation}\label{sec:Boltzmann}

We now introduce the notion of weak solutions we deal with.
We denote by $\cP_2(\rr^3)$ the set of probability measures on $\rr^3$ such that
$m_2(f):=\intrd |v|^2 f(\dd v)<\infty$.

\begin{defi}\label{dfsol}
Assume~\eqref{hyp0}-\eqref{hyp1} with some $\gamma \in (-3,0)$.
A function $f \in C([0,\infty);\cP(\rr^3))$
is a weak solution to~\eqref{be} if for all $T>0$,
\begin{gather}
\sup_{t\in [0,T]} m_2(f_t)<\infty, \label{co} \\
\int_0^T \dd t  \intrd f_t(\dd v)\intrd f_t(\dd v_*) |v-v_*|^{2+\gamma} < \infty,\label{cond}
\end{gather}
and if for any $\phi \in C^2_b(\rr^3)$ and any $t\geq 0$,
\begin{equation}\label{wbe}
\intrd \phi(v) f_t(\dd v) =  \intrd \phi(v) f_0(\dd v)
+\int_0^t \dd s \intrd f_s(\dd v) \intrd f_s(\dd v_*)  \cA\phi(v,v_*),
\end{equation}
where $\cA$ has to be understood as the principal value (here $\theta$
is defined by $\cos\theta=\frac{v-v_*}{|v-v_*|}\cdot\sigma$)
\begin{equation}\label{afini}
\cA\phi (v,v_*) = \frac12 \lim_{\eta\to 0}\int_{\Sp_2} 
[\phi(v')+\phi(v'_*)-\phi(v)-\phi(v_*)]\indiq_{\{\theta>\eta\}} B(v-v_*,\sigma) \dd \sigma.
\end{equation}
\end{defi}
We also introduce, for $\phi \in C^2_b(\rr^3)$
and $v,v_* \in \rr^3$,
\begin{equation}\label{abar}
\bar \cA \phi (v,v_*) = \int_{\Sp_2} (\phi(v')-\phi(v)-(v'-v)\cdot \nabla \phi(v)) B(v-v_*,\sigma) \dd \sigma
- b |v-v_*|^\gamma (v-v_*)\cdot \nabla \phi(v),
\end{equation}
where
\beqn\label{def:b} 
b :=\pi\int_0^\pi (1-\cos\theta)\beta(\theta)\dd \theta < \infty, 
\eeqn 
The principal value in~\eqref{afini} is useless if 
$\int_0^\pi \theta\beta(\theta)\dd\theta<\infty$. However, assuming only~\eqref{hyp1}, we have no much than
$\int_0^\pi \theta^2\beta(\theta)\dd\theta<\infty$ which is a weaker information. 
As shown by Villani~\cite{MR1650006}, see also~\eqref{est5} below, we have 
$|\cA\phi (v,v_*)|\leq C_\phi |v-v_*|^{2+\gamma}$ for any $\phi \in C^2_b(\rr^3)$, so that everything 
makes sense in the above definition.  

\begin{thm}\label{thm:exist&uniqBoltz} Assume~\eqref{hyp0}-\eqref{hyp1} with $\gamma \in (-2,0)$ 
and that $f_0 \in \cP_2(\rr^3)$ has a finite Fisher information.
Then, there exists a unique weak solution $f \in C([0,\infty);L^1(\rr^3))$ to the Boltzmann equation~\eqref{be} 
in the sense of Definition~\ref{dfsol} and such that $t \mapsto I_1(f_t) \in L^1_{loc}([0,\infty))$.
\end{thm}

\begin{proof}
The existence of weak solutions has been proved by Villani~\cite{MR1650006}.
His proof is based on a weak compactness argument in $L^1(\rr^3)$, a clever use of the so-called 
dissipation of entropy term and some accurate estimates on the function $\cA\phi$. 
On the other hand, as already mentioned, Imbert, Silvestre and Villani in~\cite{imbertSV2024} have established 
that the Fisher information $I_1(f_t)$ is decaying in such a situation and thus locally time-integrable.
The uniqueness of a weak solution lying in $L^1_{loc}(\rr_+, L^p(\rr^3))$, for some $p\in (3/(3+\gamma),\infty)$,
can be found in~\cite[Corollary~1.5 and Lemma~1.1-(i)]{MR2398952}. 
Since $\|f\|_{L^3(\rr)} \leq C_0 I_1(f)$ by Lemma~\ref{pfish}-(b) and since 
$3 \in  (3/(3+\gamma),\infty)$ because of the restriction $\gamma \in (-2,0)$, the conclusion follows.
\end{proof}

We end this section by presenting some estimates on $\cA$ and $\tilde \cA$.

\begin{lemma}\label{est} Assume~\eqref{hyp0}-\eqref{hyp1} with some $\gamma \in (-3,0)$.
Recall that $v'=v'(v,v_*,\sigma)$ and $v'_*=v'_*(v,v_*,\sigma)$ were introduced in~\eqref{vprimeetc}, for any 
$v,v_* \in \rr^3$ and $\sigma \in \Sp_2$. We have
\begin{gather}
v'+v'_*=v+v_*, \qquad  |v'|^2+|v'_*|^2=|v|^2+|v_*|^2, \label{est1}\\
|v'-v|+|v'_*-v_*|\leq \theta |v-v_*|. \label{est2}
\end{gather}
We have
\begin{align}
\lim_{\eta\to 0} \int_{\Sp_2}(v'-v) B(v-v_*,\sigma)\indiq_{\{\theta>\eta\}} \dd \sigma =& 
-\lim_{\eta\to 0} \int_{\Sp_2}(v'_*-v_*) B(v-v_*,\sigma)\indiq_{\{\theta>\eta\}} \dd \sigma \label{est3}\\
=& b|v-v_*|^\gamma(v_*-v),\notag\\
\int_{\Sp_2}|v'-v|^2 B(v-v_*,\sigma) \dd \sigma =& \int_{\Sp_2} |v'_*-v_*|^2 B(v-v_*,\sigma) \dd \sigma=
b |v-v_*|^{\gamma+2}. \label{est4}
\end{align}
For all $\phi \in C^2_b(\rr^3)$, recalling~\eqref{afini} and~\eqref{abar}, 
there is a constant $C_\phi \in (0,\infty)$ such that
\begin{gather}\label{est5}
|\cA \phi(v,v_*)|\leq C_\phi |v-v_*|^{2+\gamma},\\
\label{est5p1}
|\bar\cA\phi(v,v_*)|\leq C_\phi (|v-v_*|^{1+\gamma}+|v-v_*|^{2+\gamma}),\\
\label{est6}
\cA\phi(v,v_*)=\frac12[\bar\cA\phi(v,v_*) +\bar \cA\phi(v_*,v)].
\end{gather}
\end{lemma}

\begin{proof} The proof follows closely  Villani~\cite{MR1650006}.
We fix $v\neq v_*$ (if $v=v_*$, then $v'=v$ and $v'_*=v_*$, so that everything is obvious).
The identities~\eqref{est1} are immediate from~\eqref{vprimeetc}. We next introduce some polar coordinate 
system with axis $v-v_*$ and write, with $\theta\in [0,\pi)$ and $\varphi\in [0,2\pi)$,
$$
\sigma=\frac{v-v_*}{|v-v_*|}\cos\theta+(\cos\varphi \bi + \sin \varphi \bj)\sin \theta,
\;\; \text{so that} \;\; B(v-v_*,\sigma)\dd \sigma = |v-v_*|^\gamma \beta(\theta)\dd \theta \dd \varphi. 
$$
Setting $\Gamma(\varphi)=|v-v_*|(\cos\varphi \bi + \sin \varphi \bj)$, the identities~\eqref{vprimeetc} write 
$$
v'=v-\frac{1-\cos\theta}2(v-v_*)+\frac{\sin\theta}2\Gamma(\varphi)\quad \text{and} \quad
v'_*=v_*+\frac{1-\cos\theta}2(v-v_*)-\frac{\sin\theta}2\Gamma(\varphi).
$$
One thus has $|v'-v|^2=|v'_*-v_*|^2=\frac{1-\cos\theta}2|v-v_*|^2 = \sin^2(\theta/2) |v-v_*|^2$, 
from which both~\eqref{est2} and~\eqref{est4} follow.
Using that
$$
\int_0^{2\pi} (v'-v) \dd\varphi=-\int_0^{2\pi} (v'_*-v_*) \dd\varphi= \pi(1-\cos\theta)(v_*-v),
$$
integrating these identities against $|v-v_*|^\gamma \indiq_{\{\theta>\eta\}}\beta(\theta)\dd \theta$ and 
letting $\eta\to 0$, one gets~\eqref{est3}.
Next, performing a Taylor expansion as in~\cite[Section~4]{MR1650006}, we have 
$$ 
\Bigl| \int_0^{2\pi}(\phi(v')+\phi(v'_*)-\phi(v)-\phi(v_*))\dd\varphi \Bigr|\leq C_\phi \theta^2 |v-v_*|^2.
$$
Coming back to the definition~\eqref{afini} and recalling \eqref{hyp0}-\eqref{hyp1}, we deduce 
$$
|\cA\phi(v,v_*)|= \lim_{\eta\to 0} \frac{|v-v_*|^\gamma}2\Big|\int_\eta^\pi \int_0^{2\pi}
(\phi(v')+\phi(v'_*)-\phi(v)-\phi(v_*))\dd\varphi \beta(\theta)\dd \theta\Big| \leq C_\phi |v-v_*|^{\gamma+2}.
$$
Performing a  Taylor expansion in the very definition~\eqref{abar} of $\bar\cA\phi$, we get
$$
|\bar\cA \phi(v,v_*)|\leq ||D^2\phi||_\infty \int_{\Sp_2} |v'-v|^2B(v-v_*,\sigma)\dd \sigma
+ b  ||D\phi||_\infty|v-v_*|^{\gamma+1}, 
$$
and thus~\eqref{est5p1} follows from~\eqref{est4}. Finally, \eqref{est6} follows from~\eqref{est3}.
\end{proof}

\section{Kac $N$-particle system}\label{sec:Kac}
We reformulate the Kac $N$-particle system introduced in Section~\ref{subsec:KacSystem} 
in a more precise way, as a solution to a stochastic differential equation. 
We fix $N\geq 2$ and consider a family of independent Poisson measures 
$(\Pi^N_{ij}(\dd s,\dd \sigma,\dd z))_{1\leq i \neq j \leq N}$ on $\rr_+\times \Sp_2\times \rr_+$ with intensity measure
$\frac{1}{2(N-1)} \dd s \dd \sigma \dd z$. 
Consider $f_0 \in \cP_2(\rr^3)$ and an i.i.d. collection $(V_0^i)_{i=1,\dots,N}$ of $f_0$-distributed random
variables independent of the above Poisson measures. Set $\bV_0^N=(V_0^1,\dots,V_0^N)$ and consider
the S.D.E. with unknown $(\bV^N_t=(V^{1,N}_t,\dots,V^{N,N}_t))_{t\geq 0}$ 
valued in $\rdN$:
\begin{align}\label{ks}
\bV^N_t=&\bV^N_0 + \sum_{1\leq i\neq j \leq N} \int_0^t \int_{\Sp_2}\int_0^\infty [v'_{ij}(\bV^N_{s-},\sigma)-\bV^N_\sm]
\indiq_{\{z< B(V^{i,N}_\sm-V^{j,N}_\sm,\sigma)\}}
\tilde \Pi^N_{ij}(\dd s, \dd \sigma, \dd z) \notag \\
&+ \frac{b}{2(N-1)}\sum_{1\leq i \neq j \leq N}  \int_0^t |V^{i,N}_s - V^{j,N}_s|^\gamma (V^{i,N}_s - V^{j,N}_s)(\be_j
-\be_i ) \dd s.
\end{align}
We recall that for $h \in \rr^3$, $h\be_i=(0,\dots,0,h,0,\dots,0)$ with $h$ at the $i$-th place, that
$b$ is defined in~\eqref{def:b} and that
\beqn\label{def:tildePiN}
\tilde \Pi^N_{ij}(\dd s, \dd \sigma,\dd z)=\Pi^N_{ij}(\dd s, \dd \sigma,\dd z)
-\frac{1}{2(N-1)} \dd s \dd \sigma \dd z
\eeqn 
is the compensated Poisson measure. 
If assuming that 
$\int_0^\pi \theta\beta(\theta)\dd \theta<\infty$, one may rewrite~\eqref{ks}
in a simpler form: replace $\tilde \Pi^N_{ij}$ by $\Pi^N_{ij}$ and remove the second line.
A solution to~\eqref{ks} has to be
càdlàg and adapted to some filtration in which $(\Pi^N_{ij})_{1\leq i < j \leq N}$ are Poisson.
We recall that exchangeable random vectors and symmetric probability measures have been defined in  
Section~\ref{subsec:propgation-chaos}.
For $F$ symmetric, we introduce the energy 
\begin{gather*}
E_N(F)=\frac 1 N \int_{\rdN} |v|^2 F(v) \dd v= \int_{\rdN} |v_1|^2 F(v) \dd v
\end{gather*}
and we recall the definition of the Fisher information~\eqref{eq:FisherN} 
\begin{gather*}
I_N(F)=\frac 1 N \int_{\rdN} \frac{|\nabla F(v)|^2}{F(v)} \dd v =  \int_{\rdN} \frac{|\nabla_1 F(v)|^2}{F(v)} 
\dd v. 
\end{gather*}
Here $\nabla$ stands for the gradient in $(\rr^3)^N$ and $\nabla_1$ stands for the gradient 
in the first variable in $\rr^3$. 

\begin{thm}\label{mr1}
Assume~\eqref{hyp0}-\eqref{hyp1} for some $\gamma \in (-2,0)$.
Consider a probability density $f_0$ on $\rr^3$ such that $m_2(f_0)$ and $I_1(f_0)$ are finite. 
For each $N\geq 2$, 
there exists an exchangeable solution $(\bV^N_t)_{t\geq 0}$ to~\eqref{ks} enjoying the following properties.
First, it is strongly conservative:
\begin{equation}\label{sc}
\text{a.s., for all $t\geq 0$,} \quad  \sum_{i=1}^N V^{i,N}_t=\sum_{i=1}^N V^{i,N}_0 \quad \text{and} \quad 
\sum_{i=1}^N |V^{i,N}_t|^2=\sum_{i=1}^N |V^{i,N}_0|^2.
\end{equation}
Next, for each $t\geq 0$, denoting by $F^N_t$ the law of $\bV^N_t$,
\begin{equation}\label{eq:EnergyN}
E_N(F^N_t)=E_N(F^N_0)=m_2(f_0)\quad \text{and} \quad I_N(F^N_t) \leq I_N(F^N_0)=I_1(f_0).
\end{equation}
Moreover, $(F^N_t)_{t\geq 0}$ is a weak solution to~\eqref{eq:BlNeq} in the following sense:
$$
\int_{\rdN} \phi(v) F^N_t(\dd v) = \int_{\rdN} \phi(v) F^N_0(\dd v)+ \int_0^t 
\int_{\rdN} \cL_N \phi(v) F^N_s(\dd v) \dd s, 
$$
for all $\phi \in C^2_b(\rdN)$ and for all $t\geq 0$. 
\end{thm}

It is worth emphasizing that $\cL_N\phi$ is well-defined as the principal value 
\begin{equation}\label{LN}
\cL_N \phi(v)= \frac 1 {2(N-1)} \sum_{1\leq i \neq  j \leq N} \lim_{\eta\to 0}\int_{\Sp_2} 
[\phi(v'_{ij}) - \phi(v)] B(v_i-v_j,\sigma)\indiq_{\{\theta_{ij}>\eta\}} \dd \sigma.
\end{equation}
for all $\phi\in C^2_b(\rdN)$, where $\theta_{ij}\in [0,\pi]$ is defined by
$\cos\theta_{ij}= \frac{v_i-v_j}{|v_i-v_j|}\cdot\sigma$.
This precaution is useless if the singularity of $\beta$ is removed. 
Using \eqref{vvvv}, one can check that $\cL_N$ is self-adjoint:  for
$\phi,\psi:\rdN\to \rr$ smooth enough, it holds 
\begin{equation}\label{edpF}
\int_{\rdN} \phi(v) \cL_N \psi (v) \dd v = \int_{\rdN} \psi(v) \cL_N \phi (v) \dd v.
\end{equation}

We now state without proof the immediate $N$-particle counterpart of Lemma~\ref{est}.

\begin{lemma}\label{estBIS} Assume~\eqref{hyp0}-\eqref{hyp1} with some $\gamma \in (-2,0)$.
Recall that $v'_{ij}=v'_{ij}(v,\sigma)\in\rdN$ was defined in~\eqref{vpij} for any $v \in \rdN$ 
and $\sigma \in \Sp_2$. Denoting by $v'_{ijk}$ the $k$-th coordinate of $v'_{ij}$ and introducing $\theta_{ij}$
as a few lines above, it holds 
\begin{gather}
\sum_{k=1}^Nv'_{ijk}(v,\sigma) =\sum_{k=1}^N v_k, \qquad \sum_{k=1}^N|v'_{ijk}(v,\sigma)|^2=\sum_{k=1}^N|v_k|^2 ,
\label{est7}\\
|v'_{ij}-v|\leq \theta_{ij}|v_i-v_j| \label{est8}
\end{gather}
and, with $b$ defined in~\eqref{def:b},
\begin{gather}
\lim_{\eta\to 0}\int_{\Sp_2}[v'_{ij}(v,\sigma)-v]B(v_i-v_j,\sigma)\indiq_{\{\theta_{ij}>\eta\}} 
\dd \sigma = b |v_i-v_j|^\gamma(v_i-v_j)(\be_j - \be_i), 
\label{est9}\\
\int_{\Sp_2}|v'_{ij}(v,\sigma)-v|^2 B(v_i-v_j,\sigma) \dd \sigma = 2b |v_i-v_j|^{\gamma+2}. \label{est10}
\end{gather}
For all $\phi \in C^2_b(\rdN)$, there is a constant $C_{N,\phi}\in (0,\infty)$ such that
\begin{equation}\label{est11}
|\cL_N \phi(v)|\leq C_{N,\phi} \sum_{1\leq i \neq j \leq N} |v_i-v_j|^{\gamma+2}.
\end{equation}
Also, $\cL_N \phi= \bar \cL_N\phi$ for any $\phi \in C^2_b(\rdN$, where
\begin{align}\label{est12}
\bar \cL_N \phi(v)=&\frac 1{2(N-1)}\sum_{1\leq i \neq j \leq N} \int_{\Sp_2} 
[\phi(v'_{ij}(v,\sigma)) - \phi(v)-(v'_{ij}(v,\sigma)-v)\cdot \nabla\phi(v)] B(v_i-v_j,\sigma) \dd \sigma\notag\\
& + \frac b{2(N-1)}\sum_{1\leq i \neq j \leq N} |v_i-v_j|^\gamma [(v_i-v_j)(\be_j-\be_i)]\cdot \nabla\phi(v).
\end{align}
\end{lemma}  

We are now in position to present the proof of Theorem~\ref{mr1}.

\begin{proof}[Proof of Theorem~\ref{mr1}]
We fix $N\geq 2$, $B$ satisfying~\eqref{hyp0}-\eqref{hyp1} 
with $\gamma \in (-2,0)$, as well as 
$f_0\in \cP_2(\rr^3)$ such that $I_1(f_0)<\infty$. We 
introduce a family of smooth probability densities $f_0^\e$,
  weakly  converging to $f_0$ and such that $m_2(f_0^\e)\leq 2 m_2(f_0)$, 
$I_1(f_0^\e) \leq 2 I_1(f_0)$, $\lim_\e m_2(f^\e_0)=m_2(f_0)$ 
and $\lim_\e I_1(f_0^\e)=I_1(f_0)$. 
Denoting by $\mathfrak{g}$ the standard Gaussian density, we may also assume  that 
$\e \mathfrak{g}  \le f_0^\e \le \e^{-1} \mathfrak{g}$ and $ \mathfrak{g}^{-1/2} \nabla f_0^\e \in L^2(\rr^3)$, 
for any given $\e\in (0,\e_0)$, $\e_0 > 0$ small enough 
(choose $f_0^\e := [ (\e \mathfrak g_{1-\e} + f_0) \wedge (\e^{-1} \mathfrak g_{1-\e})] *  \mathfrak g_{\e}$, where 
$\mathfrak g_{\theta}$,
is the centered Gaussian density with variance $\theta$).
Consider, for each $\e \in (0,\e_0)$, the cross section
 $B_\e(z,\sigma)\sin \theta=(\e^2+|z|^2)^{\gamma/2}(\beta(\theta)\land \e^{-1})$,
which satisfies
$$
M_\e:=\sup_{z \in \rr^3, \sigma \in \Sp_2} B_\e(z,\sigma)<\infty.
$$
Consider also an i.i.d. collection of Poisson measures 
$(\Pi^N_{ij}(\dd s,\dd \sigma,\dd z))_{1\leq i \neq j \leq N}$ on $\rr_+\times \Sp_2\times \rr_+$ with intensity measure
$\frac{1}{2(N-1)} \dd s \dd \sigma \dd z$ and 
an i.i.d. collection $(V_0^{i,\e})_{i=1,\dots,N}$ of $f_0^\e$-distributed random
variables independent of the Poisson measures. Set $\bV_0^{N,\e}=(V_0^{1,\e},\dots,V_0^{N,\e})$ and consider
the stochastic differential equation
\begin{align*}
\bV^{N,\e}_t=&\bV^{N,\e}_0 + \sum_{1\leq i\neq j \leq N} \int_0^t \int_{\Sp_2}\int_0^\infty 
[v'_{ij}(\bV^{N,\e}_{s-},\sigma)-\bV^{N,\e}_\sm]
\indiq_{\{z< B_\e(V^{i,N,\e}_\sm-V^{j,N,\e}_\sm,\sigma)\}}\Pi^N_{ij}(\dd s, \dd \sigma, \dd z).
\end{align*}
We can replace the integral $\int_0^\infty$ by $\int_0^{M_\e}$, so that this equation involves finite 
Poisson measures and is thus strongly well-posed: it can be solved by induction on the jump instants
of the Poisson measures.
We deduce from~\eqref{est7} that $(\bV^{N,\e}_t)_{t\geq 0}$ 
is strongly conservative: it satisfies~\eqref{sc}. This readily implies that
\begin{equation}\label{upn}
\sup_{t\geq 0} \E[|V^{1,N,\e}_t|^2]=m_2(f_0^\e)\leq 2 m_2(f_0).
\end{equation}
Using the Itô formula, we get,
for all measurable $\phi:\rdN\to \rr$,
\begin{align*}
\phi(\bV^{N,\e}_t)=&\phi(\bV^{N,\e}_0) + \sum_{1\leq i\neq j \leq N} \int_0^t \int_{\Sp_2}\int_0^\infty 
[\phi(v'_{ij}(\bV^{N,\e}_{s-},\sigma))-\phi(\bV^{N,\e}_\sm)]\\
& \hskip5cm \indiq_{\{z< B_\e(V^{i,N,\e}_\sm-V^{j,N,\e}_\sm,\sigma)\}}\Pi^N_{ij}(\dd s, \dd \sigma, \dd z).
\end{align*}
We now may take expectations if $\phi$ is bounded and get,
denoting by $F^{N,\e}_t$ the law of $\bV^{N,\e}_t$,
$$
\int_\rdN \phi(v) F^{N,\e}_t(\dd v) = \int_\rdN \phi(v) F^{N,\e}_0(\dd v) + \int_0^t \int_\rdN \cL_{N,\e} 
\phi (v) F^{N,\e}_s(\dd v) \dd s,
$$
with $\cL_{N,\e}$ defined as in~\eqref{LN}, replacing $B$ by $B_\e$ (here of course the principal 
value can be removed). Using  \eqref{edpF},
we immediately obtain that $(F^{N,\e}_t)_{t\geq 0}$ is a weak solution to 
$$
\partial_t F^{N,\e}_t = \cL_{N,\e}F^{N,\e}_t.
$$
The equation is linear with smooth coefficients and satisfies a weak maximum principle, from what we 
classically deduce that the solution $F^{N,\e}_t$ satisfies 
$\e^N \mathfrak{g}^{\otimes N}  \le F^{N,\e}_t  \le \e^{-N} \mathfrak g^{\otimes N}$ and 
$ (\mathfrak g^{-1/2})^{\otimes N} \nabla F^{N,\e}_t \in L^2(\rr^{3N})$,  uniformly in $t \in [0,T]$, 
for any given $N\geq 2$, $\e\in (0,\e_0]$, $T > 0$. The function $F^{N,\e}$ is thus smooth enough 
in order to justify the chain 
rule used during the proof of Proposition~\ref{prop:decay}. 
As a consequence, applying Proposition~\ref{prop:decay} as well as  Lemma~\ref{pfish}-(a), we get
\begin{equation}\label{deceps}
I_N(F^{N,\e}_t)\leq I_N(F^{N,\e}_0)=I_1(f_0^\e) \leq 2 I_1(f_0).
\end{equation} 
By Lemma~\ref{pfish}-(d), this implies that for all $a \in (-2,0)$,
\begin{equation}\label{momi}
\sup_{t\geq 0} \E[|V^{1,N,\e}_t-V^{2,N,\e}_t|^a]\leq 2 C_a I_1(f_0).
\end{equation}
Now, set $b_\e=\pi\int_0^\pi (1-\cos \theta) (\beta(\theta)\land\e^{-1})\dd \theta$  and observe that,
using (some $\e$-version of)~\eqref{est9},
\begin{align*}
\bV^{N,\e}_t=&\bV^N_0 + \sum_{1\leq i\neq j \leq N} \int_0^t \int_{\Sp_2}\int_0^\infty 
[v'_{ij}(\bV^{N,\e}_{s-},\sigma)-\bV^{N,\e}_\sm]
\indiq_{\{z< B_\e(V^{i,N,\e}_\sm-V^{j,N,\e}_\sm,\sigma)\}}\tilde \Pi^N_{ij}(\dd s, \dd \sigma, \dd z)\\
& + \frac{b_\e}{2(N-1)}\sum_{1\leq i \neq j \leq N}  \int_0^t (\e^2+|V^{i,N,\e}_s - V^{j,N,\e}_s|^2)^{\gamma/2} 
(V^{i,N,\e}_s - V^{j,N,\e}_s)(\be_j -\be_i ) \dd s, 
\end{align*}
where we recall that $\tilde \Pi^N_{ij}$ is defined in~\eqref{def:tildePiN}. 
Observe that for all $v \in \rdN$,
$$
(\e^2+|v_i - v_j|^2)^{\gamma/2} |v_i-v_j|\leq |v_i-v_j|^{\gamma+1}
$$ 
and, by~\eqref{est10},
\begin{gather*}
\int_{\Sp_2} |v'_{ij}(v,\sigma)-v|^2B_\e(v_i-v_j,\sigma) \dd \sigma
\leq 2b_\e  |v_i-v_j|^{\gamma+2} \leq 2b  |v_i-v_j|^{\gamma+2}.
\end{gather*}
Since $-2<1+\gamma<2+\gamma<2$, 
the tightness of the family $((\bV^{N,\e}_t)_{t\geq 0}, \e \in (0,\e_0))$ is easily deduced from~\eqref{momi} 
(using also exchangeability) and~\eqref{upn}.
See the proof Lemma~\ref{tight} for a very similar argument.
It is then standard to pass to the limit as $\e\to 0$ and to show the existence, 
through martingale problems, of a solution $(\bV^{N}_t)_{t\geq 0}$ to~\eqref{ks}.
 See the proof of Theorem~\ref{mr2}, of which we give all the details, for a similar but more complicated
demonstration.
The solution $(\bV^{N}_t)_{t\geq 0}$ of course inherits exchangeability,  
strong conservation~\eqref{sc}, which obviously implies that
$E_N(F^N_t)=E_N(F^N_0)=m_2(f_0)$, where $F^N_t$ is the law of $\bV^N_t$. We also have
$$
I_N(F^N_t)\leq \liminf_\e I_N(F^{N,\e}_t)\leq \liminf_\e I_1(f_0^\e)=I_1(f_0)
$$ 
by Lemma~\ref{pfish}-(e) and~\eqref{deceps}. Finally, using the Itô formula for~\eqref{ks} 
and taking expectations, we find that for all $\phi\in C^2_b(\rdN)$, all $t\geq 0$,
\begin{equation}
\int_\rdN \phi(v) F^N_t(\dd v) = \int_\rdN \phi(v) F^N_0(\dd v) + \int_0^t \int_\rdN \bar \cL_N\phi(v) F^N_s(\dd v)
\dd s, 
\end{equation}
where for $v \in \rdN$, $ \bar \cL_N\phi (v)$ is defined by~\eqref{est12}. 
We have $\cL_N\phi=\bar \cL_N\phi$ by Lemma~\ref{est}, and the proof is complete.
\end{proof}

\section{The nonlinear stochastic equation}\label{sec:NonlinearStochastic}

We introduce some stochastic process $(V_t)_{t\geq 0}$ valued in $\rr^3$, representing
the time-evolution of the velocity of a typical particle in the gas. This was initiated by 
Tanaka~\cite{MR512334} for Maxwell molecules and by Sznitman~\cite{MR753814} for more general cross sections.
This process should solve the jumping S.D.E. 
\begin{align}\label{edsnl}
V_t=&V_0+\int_0^t\int_{\rr^3}\int_{\Sp_2}\int_0^\infty (v'(V_\sm,v_*,\sigma)-V_\sm)
\indiq_{\{z<B(V_\sm-v_*,\sigma)\}}\tilde \Pi(\dd s, \dd v_*,\dd \sigma, \dd z)\\
&-b\int_0^t \int_{\rr^3}|V_s-v_*|^\gamma(V_s-v_*) f_s(\dd v_*) \dd s,\notag
\end{align}
with $f_t$ the law of $V_t$ and with $\Pi(\dd s, \dd v_*\dd \sigma, \dd z)$ a Poisson measure
on $\rr_+\times \rr^3\times \Sp_2\times \rr_+$ with intensity $\dd s f_s(\dd v_*) \dd \sigma \dd z$.
This equation is nonlinear in that the time marginals of the solution appear in its dynamics.

\vip

Actually, we will not directly deal with~\eqref{edsnl}, but rather with the corresponding
nonlinear martingale problem. We introduce, on the canonical space $\Omega^*=\D(\rr_+,\rr^3)$,
the canonical process $(V^*_t)_{t\geq 0}$ defined by $V^*_t(w)=w(t)$. We endow $\Omega^*$
with the Skorokhod topology and the associated $\sigma$-field $\cF^*=\sigma(V^*_t,t\geq 0)$
and filtration $\cF^*_t=\sigma(V_s^*,s\in [0,t])$. 

\vip

\begin{defi}\label{dfmp}
Assume that $B$ satisfies~\eqref{hyp0}-\eqref{hyp1} with some $\gamma\in (-3,0)$
and consider $f_0 \in \cP_2(\rr^3)$.
For  a probability measure $\QQ$ on $(\Omega^*,\cF^*,(\cF^*_t)_{t\geq 0})$ and $t\geq 0$, we set
$\QQ_t(\dd v)=\QQ(V^*_t \in \dd v)$. We say that 
$\QQ$ solves $MP(f_0)$ if  the following conditions hold: $\QQ_0=f_0$,  for all $T>0$,
\begin{gather}
\sup_{t\in [0,T]} m_2(\QQ_t)<\infty, \label{cop}\\
\int_0^T \int_{\rr^3}\int_{\rr^3} |v-v_*|^{\gamma+1} \QQ_s(\dd v_*)\QQ_s(\dd v)\dd s<\infty, \label{condp}
\end{gather}
and for all $\phi \in C^2_b(\rr^3)$, the process
$$
M^\phi_t=\phi(V^*_t)-\phi(V^*_0)-\int_0^t \int_{\rr^3}\bar \cA \phi(V^*_s,v_*) \QQ_s(\dd v_*)\dd s
$$
is a martingale under $\QQ$.
\end{defi}

Because of the estimate~\eqref{est5p1} and the condition $\gamma+1<\gamma+2<2$, it holds  
$$
\E_\QQ\Big[\int_0^t\int_{\rr^3} |\bar \cA \phi(V^*_s,v_*)| \QQ_s(\dd v_*) \dd s\Big]
=\int_0^t \int_{\rr^3}\int_{\rr^3} |\bar \cA \phi(v_s,v_*)| \QQ_s(\dd v_*) \QQ_s(\dd v)  \dd s 
<\infty,
$$ 
so that this definition makes sense.
A solution to the nonlinear martingale problem provides a solution to the Boltzmann equation. 

\begin{prop}\label{ev}
Assume that $B$ satisfies~\eqref{hyp0}-\eqref{hyp1} with some $\gamma\in (-3,0)$ 
and consider $f_0 \in \cP_2(\rr^3)$. For $\QQ$ solving $MP(f_0)$, $(\QQ_t)_{t\geq 0}$ is a weak solution
to the Boltzmann equation~\eqref{be} starting from $f_0$.
\end{prop}

\begin{proof}[Proof of Proposition~\ref{ev}] 
From~\eqref{cop},~\eqref{condp} and  $\gamma+1<\gamma+2<2$, we have that  $(\QQ_t)_{t\geq 0}$
satisfies~\eqref{co} and~\eqref{cond}.
For $\phi \in C^2_b(\rr^3)$ and $t\geq 0$, we have $\E[M^\phi_t]=0$, 
i.e.
$$
\int_{\rr^3}\phi(v)\QQ_t(\dd v) - \int_{\rr^3}\phi(v)\QQ_0(\dd v) - \int_0^t  \int_{\rr^3} \int_{\rr^3}
\bar\cA\phi(v,v_*) \QQ_s(\dd v_*)\QQ_s(\dd v)\dd s=0.
$$
But  $\int_{\rr^3} \int_{\rr^3}\bar\cA\phi(v,v_*) \QQ_s(\dd v_*)\QQ_s(\dd v)= \int_{\rr^3} \int_{\rr^3}
\cA\phi(v,v_*) \QQ_s(\dd v_*)\QQ_s(\dd v)$ by~\eqref{est6}, which completes the proof.
\end{proof}

We also have the following uniqueness result.

\begin{prop}\label{unimp}
Assume that $B$ satisfies~\eqref{hyp0}-\eqref{hyp1} for some $\gamma\in (-2,0)$ and that 
$f_0 \in \cP_2(\rr^3)$ has a finite Fisher information. Then, there is at most one
solution $\QQ$ to $MP(f_0)$ such that $t \mapsto I_1(\QQ_t) \in L^1_{loc}([0,\infty))$.
\end{prop}
 
\begin{proof}[Proof of Proposition~\ref{unimp}] 
Consider two solution $\QQ$, $\QQ'$ to $MP(f_0)$
such that the Fisher informations of both $(\QQ_t)_{t\geq 0}$ and $(\QQ'_t)_{t\geq 0}$ are locally integrable. 
By Proposition~\ref{ev} and Theorem~\ref{thm:exist&uniqBoltz}, we conclude that $\QQ_t=\QQ'_t$ for all $t\geq 0$.
\vip
Consequently, both $\QQ$ and $\QQ'$ solve the {\it linear} martingale problem $MP((\tilde\cA_t)_{t\geq 0},f_0)$, in
the notation of~\cite[proof of Lemma~4.6, Step~1]{MR2398952} where for $\phi \in C^2_c(\rr^3)$, $t\geq 0$ 
and $v\in \rr^3$,
$$
\tilde\cA_t \phi(v)= \int_{\rr^3} \bar\cA\phi(v,v_*) \QQ_t(\dd v_*).
$$
Note that $\tilde\cA \phi$ defined in~\cite[Eq.~(6.20)]{MR2398952}
equals our $\bar \cA \phi$. Using Lemma~\ref{pfish}-(b), we observe that 
$$
(\QQ_t)_{t\geq 0} \in L^\infty_{loc}(\rr_+,\cP_2(\rr^3))\cap L^1_{loc}(\rr_+,L^3(\rr^3))\subset 
L^\infty_{loc}(\rr_+,\cP_2(\rr^3))\cap L^1_{loc}(\rr_+,J_\gamma),
$$
see~\cite[Eq.~(1.12) and~(5.2)]{MR2398952} and use that $3>3/(3+\gamma)$. It is worth emphasizing that it is 
here we  really  need the  moderately soft potentials restriction  $\gamma \in (-2,0)$.
It is shown in~\cite[proof of Lemma~4.6, Steps~3-4-5]{MR2398952} that
under such a condition,  $MP((\tilde\cA_t)_{t\geq 0},f_0)$ has a most one solution, when $f_0$ is a Dirac mass,
but the proof uses only that $m_2(f_0)<\infty$. All in all, $\QQ=\QQ'$.
\end{proof}

\section{From the Kac system to the nonlinear stochastic equation}\label{sec:KACtoNL}

In this section, we mix ideas taken from the seminal papers of Sznitman~\cite{MR1108185,MR753814} and
Méléard~\cite{MR1431299} on the chaos issue for the Boltzmann equation with ideas coming from \cite{MR3254330} 
that concerns 
singular interactions (for continuous paths).
We aim to establish that the $N$-particles stochastic 
trajectories converge in law to solutions of the nonlinear martingale problem.
Below, $\cP(\D(\rr_+,\rr^3))$ is endowed with the
weak convergence topology, $\D(\rr_+,\rr^3)$ being endowed with the Skorokhod topology. We refer to 
Jacod-Shiryaev~\cite[Section~VI]{MR1943877} for a detailed account of this topology.
 We recall in particular that for each $t\geq 0$, the map $\pi_t : \D(\rr_+,\rr^3) \to \rr^3$ defined 
by $\pi_t(x)=x(t)$ is continuous at any point $x$ such that $\Delta x(t)=0$. 

\begin{thm}\label{mr2}
Assume~\eqref{hyp0}-\eqref{hyp1} with some $\gamma \in (-2,0)$.
Consider a probability density $f_0$ on $\rr^3$ such that $m_2(f_0)$ and $I_1(f_0)$ are finite and, 
for each $N\geq 2$, the process $(\bV^N_t)_{t\geq 0}$ built in Theorem~\ref{mr1}. 
Set $\mu^N=\frac1N \sum_{i=1}^N
\delta_{(V^{i,N}_t)_{t\geq 0}}$, which is a random probability measure on $\D(\rr_+,\rr^3)$. 
\vip
(i) The family $(\mu^N,N\geq 2)$ is tight in $\cP(\D(\rr_+,\rr^3))$.
\vip
(ii) Any (possibly random) limit point $\mu \in \cP(D(\rr_+,\rr^3))$ of $(\mu^N,N\geq 2)$ a.s.
solves $MP(f_0)$ and we have $\sup_{t\geq 0} \E[I_1(\mu_t)] \leq I_1(f_0)$.
\vip
(iii) The sequence $(\mu^N)_{N\geq 2}$ converges in probability, as $N\to \infty$, to the unique 
solution $\QQ$ of $MP(f_0)$ such that $t \mapsto I_1(\QQ_t) \in L^1_{loc}([0,\infty))$.
\end{thm}

As recalled in Subsection~\ref{subsec:propgation-chaos}, point (iii) tells us that
the family $((\bV^N_t)_{t\geq 0}, N\geq 2)$ is $\QQ$-chaotic.
We start with a tightness result.

\begin{lemma}\label{tight}
Assume~\eqref{hyp0}-\eqref{hyp1} with some $\gamma\in (-2,0)$, 
fix $f_0 \in \cP_2(\rr^3)$ such that $I_1(f_0)<\infty$ and consider, 
for each $N\geq2$, the exchangeable process $(\bV^N_t)_{t\geq 0}$ built in Theorem~\ref{mr1}.
The family $((V^{1,N}_t)_{t\geq 0}, N\geq 2)$ is tight in $\D(\rr_+,\rr^3)$.
\end{lemma}

\begin{proof}[Proof of Lemma~\ref{tight}]
We fix $T>0$ and show that $((V^{1,N}_t)_{t\in [0,T]}, N\geq 2)$ 
is tight in $\D([0,T],\rr^3)$, which classically suffices.
We have $\sup_{N\geq 2} \sup_{t\geq 0} I_N(F^N_t)\leq I_1(f_0)$ by~\eqref{eq:EnergyN}, which implies 
that for all $a\in (-2,0)$,
\begin{equation}\label{tax}
\sup_{N\geq 2} \sup_{t\geq 0} \E[|V^{1,N}_t-V^{2,N}_t|^{a}]\leq C_a I_1(f_0), 
\end{equation}
by Lemma~\ref{pfish}-(d). Moreover, thanks to~\eqref{sc}, we have
\begin{equation}\label{tax2}
\sup_{N\geq 2} \E \Big[\sup_{t\geq 0}\frac 1 N \sum_{i=1}^N |V^{i,N}_t|^2\Big]
=m_2(f_0) \quad \text{and} \quad \sup_{N\geq 2} \sup_{t\geq 0}\E [|V^{1,N}_t|^2]=m_2(f_0).
\end{equation}
Recalling~\eqref{ks}, we write $V^{1,N}_t=V^1_0+X^{N}_t+Y^{N}_t+Z^{N}_t$, where
\begin{align*}
X^{N}_t=& \sum_{j=2}^N \int_0^t \int_{\Sp_2}\int_0^\infty [v'(V^{1,N}_\sm,V^{j,N}_\sm,\sigma)-V^{1,N}_\sm] 
\indiq_{\{z<B(V^{1,N}_\sm-V^{j,N}_\sm,\sigma)\}}\tilde\Pi^N_{1j}(\dd s,\dd \sigma,\dd z) \\
Y^{N}_t=& \sum_{j=2}^N \int_0^t \int_{\Sp_2}\int_0^\infty [v'_*(V^{j,N}_\sm,V^{1,N}_\sm,\sigma)-V^{1,N}_\sm] 
\indiq_{\{z<B(V^{j,N}_\sm-V^{1,N}_\sm,\sigma)\}}\tilde\Pi^N_{j1}(\dd s,\dd\sigma,\dd z),\\
Z^{N}_t=&\frac{b}{N-1} \sum_{j=2}^N \int_0^t | V^{1,N}_s-V^{j,N}_s|^\gamma(V^{j,N}_s-V^{1,N}_s) \dd s.
\end{align*}

Let us show that $((Z^{N}_t)_{t\geq 0}, N\geq 2)$ is tight in $C([0,T],\rr^3)$.
We write, for all $0\leq s \leq t \leq T$,
$$
|Z^{N}_t-Z^{N}_s|\leq \frac b{N-1} \sum_{j=2}^N \int_s^t |V^{1,N}_u-V^{j,N}_u|^{\gamma+1} \dd u \leq 
\sqrt{t-s} \sqrt{U_{N,T}},
$$
where 
$$
U_{N,T}=\frac{b^2}{N-1}\sum_{j=2}^N\int_0^T |V^{1,N}_u-V^{j,N}_u|^{2(\gamma+1)} \dd u.
$$
By exchangeability and~\eqref{tax2} if $\gamma\in [-1,0)$ (whence $2(\gamma+1)\in [0,2)$)
or~\eqref{tax} if $\gamma\in (-2,-1)$ (whence $2(\gamma+1)\in (-2,0)$), we see that
$$
C_T:=\sup_{N\geq 2} \E[U_{N,T}] <\infty.
$$
For $A>0$, let 
$$
K_{A,T}=\Big\{x\in C([0,T],\rr^3): x(0)=0, \sup_{0\leq s<t\leq T} \frac{|x(t)-x(s)|}
{|t-s|^{1/2}} \leq A\Big\}, 
$$
which is compact in  $C([0,T],\rr^3)$ by the Ascoli theorem. It holds that
$$
\PP \Big( (Z^{N}_t)_{t\in [0,T]} \notin K_{A,T}\Big) \leq \PP(U_{N,T}^{1/2} \geq A) \leq \frac1{A^2}\E[U_{N,T}]
\leq \frac{C_T}{A^2}.
$$
Thus $\lim_{A\to \infty} \sup_{N\geq 2}\PP( (Z^{N}_t)_{t\in [0,T]} \notin K_{A,T})=0$, 
so that $((Z^{N}_t)_{t\geq 0}, N\geq 2)$ is tight in $C([0,T],\rr^3)$.
\vip

For $((X^{N}_t)_{t\geq 0}, N\geq 2)$ to be tight in $\D([0,T],\rr^3)$, it suffices, by the Aldous 
criterion~\cite{aldous}, see also Jacod-Shiryaev~\cite[Section~VI, Thm.~4.5]{MR1943877}, that
\vip
\noindent (i) $\sup_{N\geq 2} \E[\sup_{[0,T]} |X^N_t|^2] <\infty$,
\vip
\noindent (ii) $\lim_{\delta\to 0}\sup_{N\geq 2} \sup_{(S,S')\in \cA^N_{T,\delta}}\E[|X^N_{S'}-X^N_S|^2]=0$,
where $\cA^N_{T,\delta}$ is the set of couples $(S,S')$ of stopping times such that a.s.,
$0\leq S \leq S' \leq S+\delta \leq T$.

\vip
To check (i), we use Doob's inequality and exchangeability to write
\begin{align*}
\E\Big[\sup_{[0,T]} |X^N_t|^2\Big]\leq& 4\E[|X^N_T|^2]\\
=&\frac{2}{N-1} \sum_{j=2}^N \int_0^T \int_{\Sp_2} 
\E\Big[|v'(V^{1,N}_s,V^{j,N}_s,\sigma)-V^{1,N}_s|^2 B(V^{1,N}_s-V^{j,N}_s,\sigma)\Big] \dd \sigma \dd s\\
=& 2 \int_0^T \int_{\Sp_2} 
\E\Big[|v'(V^{1,N}_s,V^{2,N}_s,\sigma)-V^{1,N}_s|^2 B(V^{1,N}_s-V^{2,N}_s,\sigma)\Big] \dd \sigma \dd s.
\end{align*}
Using~\eqref{est4}, we conclude that
$$
\E\Big[\sup_{[0,T]} |X^N_t|^2\Big]\leq 2 b \int_0^T \E[|V^{1,N}_s-V^{2,N}_s|^{2+\gamma}] \dd s,
$$
so that (i) follows from~\eqref{tax2} (because $\gamma+2\in (0,2)$).

\vip
We next verify (ii). For $(S,S')\in \cA^N_{T,\delta}$, we have, by~\eqref{est4} again,
\begin{align*}
\E[|X^N_{S'}-X^N_S|^2]\leq &\frac{1}{2(N-1)} \sum_{j=2}^N \E\Big[\int_S^{S+\delta} \int_{\Sp_2} 
|v'(V^{1,N}_s,V^{j,N}_s,\sigma)-V^{1,N}_s|^2 B(V^{1,N}_s-V^{j,N}_s,\sigma) \dd \sigma \dd s\Big]\\
=& \frac{b}{2(N-1)} \sum_{j=2}^N \E\Big[\int_S^{S+\delta}
|V^{1,N}_s-V^{j,N}_s|^{\gamma+2} \dd s\Big].
\end{align*}
We use Hölder's inequality with $p=2/(2+\gamma)$, whence $p'=2/|\gamma|$, to write
\begin{align*}
\E[|X^N_{S'}-X^N_S|^2]\leq & \frac b2
\delta^{|\gamma|/2} \Big(\E\Big[\frac1{N-1}\sum_{j=2}^N \int_0^T |V^{1,N}_s-V^{j,N}_s|^2\dd s
\Big]\Big)^{(\gamma+2)/2},
\end{align*}
so that $\E[|X^N_{S'}-X^N_S|^2]\leq C_T \delta^{|\gamma|/2}$ for some constant $C_T$ by~\eqref{tax2}, and (ii) follows.
\vip

The tightness of $((Y^{N}_t)_{t\geq 0}, N\geq 2)$ in $\D([0,T],\rr^3)$ is shown similarly.
\end{proof}

We can finally give the

\begin{proof}[Proof of Theorem~\ref{mr2}] We divide the proof in several steps.
We recall that $(V^*_t)_{t\geq 0}$ is the canonical process on $\Omega^*=\D(\rr_+,\rr^3)$,
endowed with its canonical $\sigma$-field $\cF^*$ and canonical filtration $(\cF^*_t)_{t\geq 0}$.

\vip

{\it Step 1.}
By Lemma~\ref{tight} and Sznitman~\cite[Proposition~2.2]{MR1108185}, the family
$$
\Big(\mu^N=\frac1N \sum_{i=1}^N \delta_{V^{i,N}},N\geq 2\Big)
$$ 
is tight in $\cP(\D(\rr_+,\rr^3))$, where $V^{i,N}=(V^{i,N}_t)_{t\geq 0}$. This show (i).

\vip

{\it Step 2.} We next consider a (non relabeled) subsequence $\mu^N$ converging in law, in
$\cP(\D(\rr_+,\rr^3))$, to some (possible random) $\mu$ and we show that
$\mu$ a.s. solves $MP(f_0)$ and that for all $t\geq 0$, $\E[I_1(\mu_t)]\leq I_1(f_0)$. This will prove (ii).
Observe at once that for each $t\geq 0$,
$$
\mu^N_t(\dd v)=\frac1N \sum_{i=1}^N \delta_{V^{i,N}}(V^*_t \in \dd v)= \frac1N \sum_{i=1}^N \delta_{V^{i,N}_t}(\dd v).
$$
The convergence in law of $\mu^N$ to $\mu$ in $\cP(\D(\rr_+,\rr^3))$ implies that   both
$$
\mu^N\otimes\mu^N=\frac 1{N^2} \sum_{1\leq i,j\leq N} \delta_{(V^{i,N},V^{j,N})}\quad
\text{and} \quad \mu^N\odot\mu^N=\frac 1{N(N-1)} \sum_{1\leq i\neq j\leq N} \delta_{(V^{i,N},V^{j,N})}
$$
converge in law to $\mu\otimes\mu$ in $\cP(\D(\rr_+,\rr^3\times\rr^3))$.
\vip

{\it Step 2.1.} Fix $t\geq 0$. For $F^N_t$ the law of $\bV^N_t$, we have
$I_N(F^N_t)\leq I_1(f_0)$, see Theorem~\ref{mr1}. Since $\mu^N_t$ converges in law to $\mu_t$,
we deduce from 
Lemma~\ref{pfish}-(f) that $\E[I_1(\mu_t)]\leq I_1(f_0)$. 
\vip
Moreover, by Lemma~\ref{pfish}-(d)-(a),
we conclude that for any $T>0$, any $a \in (-2,0)$,
$$
\E\Big[\int_0^T \!\int_{\rr^3}\!\int_{\rr^3}\! |v-v_*|^{a} \mu_t(\dd v) \mu_t(\dd v_*)\dd t\Big]
\leq C_{a}\E\Big[\int_0^T\! (1+I_2(\mu_t^{\otimes 2}))\dd t\Big]
= C_{a}\E\Big[\int_0^T\! (1+I_1(\mu_t))\dd t\Big]<\infty.
$$

{\it Step 2.2.} We now show that $J$ is Lebesgue-null, where
$$
J=\{u\geq 0 : \PP(\mu(\Delta V^*_u\neq 0))>0\}.
$$
Recall that for all $x\in \D(\rr_+,\rr^3)$, $D(x)=\{u\geq 0 : \Delta x(u)\ne 0\}$
is at most countable, write $J=\cup_{n\in \nn_*} J_n$, where $J_n=\{u\geq 0 : \PP(\mu(u\in D(V^*))\geq 1/n\}$, and
$$
\int_0^\infty \indiq_{\{u\in J_n\}}\dd u 
\leq \int_0^\infty n  \PP(\mu(u \in D(V^*))) \dd u =n \E\Big[ \E_\mu\Big(
\int_0^\infty \indiq_{\{u\in D(V^*)\}}\dd u \Big)\Big]=0.
$$

{\it Step 2.3.} We now show~\eqref{cop} and, more precisely, that
\begin{equation*}
\text{a.s., for all $t\geq 0$,} \quad m_2(\mu_t) \leq m_2(f_0).
\end{equation*}

Fix $A>0$, $T>0$ and consider $\Phi_{A,T}: \cP(\D(\rr_+,\rr^3))\to \rr_+$ defined by 
$$
\Phi_{A,T}(\QQ)=\sup_{t\in [0,T]} \Big(\int_{\rr^3} (|v|^2\land A)\QQ_t(\dd v)-m_2(f_0)\Big)_+\land 1,
$$
which is bounded and continuous at any point $\QQ$ such that $t\mapsto \int_{\rr^3} (|v|^2\land A)\QQ_t(\dd v)$
has no jump at time $T$. But for all $T \in \rr_+\setminus J$,  
$t\mapsto \int_{\rr^3} (|v|^2\land A)\mu_t(\dd v)=\E_{\mu}[|V_t^*|^2\land A]$
a.s. has no jump at time $T$.
By definition of the convergence in law, we deduce that for all $T \in \rr_+\setminus J$,
\begin{align*}
\E[\Phi_{A,T}(\mu)]=&\lim_{N\to \infty} \E[\Phi_{A,T}(\mu^N)]\\
\leq & \lim_{N\to \infty}\E\Big[\sup_{t\in [0,T]}(m_2(\mu^N_t)-m_2(f_0))_+\land 1\Big]\\
\leq& \lim_{N\to \infty}\E\Big[\sup_{t\in [0,T]}(m_2(\mu^N_t)-m_2(\mu^N_0))_+\land 1\Big] + \lim_{N\to \infty} 
\E[(m_2(\mu^N_0)-m_2(f_0))_+\land 1 ],
\end{align*}
since $(x+y)_+\land 1 \leq x_+ \land 1 + y_+ \land 1$ for any $x,y \in\rr$.
The first limit equals $0$ by~\eqref{sc}.
Moreover, the law of large numbers tells us that $m_2(\mu^N_0)=\frac1N\sum_{i=1}^N |V^i_0|^2$
a.s. tends to $m_2(f_0)$,
so that the second limit is also $0$. Thus $\Phi_{A,T}((\mu_t)_{t\geq 0})=0$ a.s., showing
that a.s., for all $t\in [0,T]$, $\int_{\rr^3} (|v|^2\land A) \mu_t(\dd v) \leq m_2(f_0)$.
Since $A>0$ and $T\in \rr_+\setminus J$ can be chosen arbitrarily large, this completes the step.
\vip
{\it Step 2.4.} We next show~\eqref{condp}, i.e. that for all $T>0$, a.s.,
$$
\int_0^T \int_{\rr^3}\int_{\rr^3} |v-v_*|^{\gamma+1} \mu_t(\dd v) \mu_t(\dd v_*)\dd t<\infty. 
$$
If first $\gamma \in [-1,0)$, then $0\leq \gamma+1<2$ and this follows from Step~2.3. If next 
$\gamma \in (-2,-1)$, then $\gamma+1 \in (-1,0)\subset(-2,0)$ and this follows from Step~2.1.

\vip

{\it Step 2.5.} Consider a probability measure $\QQ$ on $\D(\rr_+,\rr^3)$
satisfying~\eqref{cop} and~\eqref{condp}.
Assume that for all
$\phi \in C^2_b(\rr^3)$, all $0\leq s_1\leq\dots\leq s_k\leq s \leq t$, with $s_1,\dots,s_k,s,t \in \rr_+\setminus J$,
all $\phi_1,\dots,\phi_k \in C_b(\rr^3)$, for $\Phi:\D(\rr_+,\rr^3\times \rr^3)\to \rr$ defined by
\begin{equation}\label{def-Phi}
\Phi(x,y)=\prod_{\ell=1}^k \phi_\ell(x(s_\ell))\Big[\phi(x(t))-\phi(x(s))- \int_s^t \bar\cA \phi(x(u),y(u))\dd u\Big],
\end{equation}
we have $\langle \QQ\otimes\QQ,\Phi\rangle = 0$. Then $\QQ$ solves $MP(f_0)$.
\vip
Recall Definition~\ref{dfmp}. The above conditions imply that for all 
$\phi \! \in \! C^2_b(\rr^3)$, all ${\phi_1,\dots,\phi_n \! \in \! C_b(\rr^3)}$, 
all $0\leq s_1\leq\dots\leq s_k\leq s \leq t$, with $s_1,\dots,s_k,s,t \in \rr_+\setminus J$,
\begin{equation}\label{ssc}
\E_\QQ\Big[\prod_{\ell=1}^k \phi_\ell(V^*_{s_\ell}) [M^\phi_t-M^\phi_s]\Big]=0.
\end{equation}
By density of $\rr_+\setminus J$ and since $(V^*_u)_{u\geq 0}$ and $(M^\phi_u)_{u\geq 0}$ are càdlàg,
we conclude that~\eqref{ssc} actually holds true for all $0\leq s_1\leq\dots\leq s_k\leq s \leq t$.
As a conclusion we have 
$$
\E_\QQ[M^\phi_t-M^\phi_s | \sigma(V^*_{u},u \in [0,s])]=0
$$
for all $0\leq s \leq t$: $(M^\phi_t)_{t\geq 0}$ is a martingale under $\QQ$.

\vip
{\it Step 2.6.} We now show that for all $\Phi$ as in
Step~2.4, we a.s. have  $\langle \mu\otimes\mu,\Phi\rangle = 0$.
Since $\mu$ a.s. satisfies~\eqref{cop} and~\eqref{condp} by Steps 2.3 and 2.4, this will prove
(by Step~2.5) that $\mu$ a.s. solves $MP(f_0)$ and this will complete the proof of (ii).

\vip

We first check that, 
\begin{equation}\label{jab1}
\lim_{N\to \infty}\E[\langle\mu^N\odot \mu^N,\Phi\rangle^2]=0.
\end{equation}
We apply the Itô formula to~\eqref{ks} to write 
$\phi(V^{i,N}_t)=\phi(V^{i,N}_s)+ M^{i,N}_{s,t}+O^{i,N}_{s,t}+D^{i,N}_{s,t},$
where
\begin{align*}
M^{i,N}_{s,t}=&\sum_{j\neq i} \int_s^t \int_{\Sp_2}\int_0^\infty \Big[\phi(v'(V^{i,N}_\um,V^{j,N}_\um,\sigma))
-\phi(V^{i,N}_\um)\Big] 
\indiq_{\{z<B(V^{i,N}_\um-V^{j,N}_\um,\sigma)\}}\tilde \Pi^N_{ij}(\dd u,\dd \sigma,\dd z),\\
O^{i,N}_{s,t}=&\sum_{j\neq i} \int_s^t \int_{\Sp_2}\int_0^\infty \Big[\phi(v'_*(V^{i,N}_\um,V^{j,N}_\um,\sigma))
-\phi(V^{i,N}_\um)\Big] 
\indiq_{\{z<B(V^{j,N}_\um-V^{i,N}_\um,\sigma)\}}\tilde \Pi^N_{ji}(\dd u,\dd \sigma,\dd z)
\end{align*}
and
\begin{align*}
D^{i,N}_{s,t}=&\frac 1{2(N-1)}\sum_{j\neq i} \int_s^t \int_{\Sp_2}\Big[\phi(v'(V^{i,N}_u,V^{j,N}_u,\sigma))-\phi(V^{i,N}_u)\\
&\hskip3cm-(v'(V^{i,N}_u,V^{j,N}_u,\sigma)-V^{i,N}_u)\cdot\nabla\phi(V^{i,N}_u)\Big]
B(V^{i,N}_u-V^{j,N}_u,\sigma)\dd \sigma \dd u,\\
&+\frac 1{2(N-1)}\sum_{j\neq i} \int_s^t \int_{\Sp_2}\Big[\phi(v'_*(V^{i,N}_u,V^{j,N}_u,\sigma))-\phi(V^{i,N}_u)\\
&\hskip3cm-(v'_*(V^{i,N}_u,V^{j,N}_u,\sigma)-V^{i,N}_u)\cdot\nabla\phi(V^{i,N}_u)\Big]
B(V^{j,N}_u-V^{i,N}_u,\sigma)\dd \sigma \dd u,\\
&+ \frac b {N-1} \sum_{j\neq i} \int_s^t
|V^{i,N}_u-V^{j,N}_u|^\gamma(V^{j,N}_u-V^{i,N}_u)\cdot \nabla\phi(V^{i,N}_u)\dd u.
\end{align*}
Recalling~\eqref{abar} and using that $v'_*(v,v_*,\sigma)=v'(v,v_*,-\sigma)$, one can check that
\begin{align*}
D^{i,N}_{s,t}=&\frac1{N-1}\sum_{j \neq i}\int_s^t \bar \cA \phi(V^{i,N}_u,V^{j,N}_u) \dd u.
\end{align*}
Consequently,
\begin{align*}
\langle\mu^N\odot \mu^N,\Phi\rangle=&
\frac1{N(N-1)}\sum_{1\leq i \neq j \leq N} \Big(\prod_{\ell=1}^k \phi_\ell(V^{i,N}_{s_\ell})\Big)
\Big[\phi(V^{i,N}_{t})-\phi(V^{i,N}_{s})
- \int_s^t \bar\cA \phi(V^{i,N}_u,V^{j,N}_u)\dd u\Big]\\
=& \frac 1 N \sum_{i=1}^N \Big(\prod_{\ell=1}^k \phi_\ell(V^{i,N}_{s_\ell})\Big)
\frac 1{N-1}\sum_{j\neq i}\Big[\phi(V^{i,N}_{t})-\phi(V^{i,N}_{s})
- \int_s^t \bar\cA \phi(V^{i,N}_u,V^{j,N}_u)\dd u\Big]\\
=&\frac 1 N \sum_{i=1}^N \Big(\prod_{\ell=1}^k \phi_\ell(V^{i,N}_{s_\ell})\Big)
\Big[\phi(V^{i,N}_{t})-\phi(V^{i,N}_{s}) - D^{i,N}_{s,t}\Big]\\
=&A_N+B_N,
\end{align*}
where
\begin{align*}
A_N=\frac 1 N \sum_{i=1}^N \Big(\prod_{\ell=1}^k \phi_\ell(V^{i,N}_{s_\ell})\Big) M^{i,N}_{s,t}
\quad \text{and} \quad B_N=\frac 1 N \sum_{i=1}^N \Big(\prod_{\ell=1}^k \phi_\ell(V^{i,N}_{s_\ell})\Big) O^{i,N}_{s,t}.
\end{align*}
Conditioning with respect to $\cF^N_s$ (where $(\cF^N_u)_{u\geq 0}$ is the canonical filtration generated by
the random variables $V^i_0$ and the Poisson measures $\Pi^N_{ij}$) and
using that the Poisson measures are independent, so that the martingales $M^{i,N}_{s,t}$ are orthogonal,
we get 
\begin{align*}
\E[A_N^2]= \frac{1}{N^2}\sum_{i=1}^N
\E\Big[\Big(\prod_{\ell=1}^k \phi_\ell(V^{i,N}_{s_\ell})\Big)^2\E\Big[(M^{i,N}_{s,t})^2\Big|\cF_{s}^N\Big]\Big]
\leq \frac{C_\Phi}{N^2}\sum_{i=1}^N \E\Big[(M^{i,N}_{s,t})^2\Big]=\frac{C_\Phi}{N}\E\Big[(M^{1,N}_{s,t})^2\Big]
\end{align*}
by exchangeability. But, using again the independence of the Poisson measures,
\begin{align*}
\E\Big[(M^{1,N}_{s,t})^2\Big]=& \frac{1}{2(N-1)}\sum_{j=2}^N \E\Big[\int_s^t \int_{\Sp_2} 
\Big[\phi(v'(V^{1,N}_u,V^{j,N}_u,\sigma))-\phi(V^{1,N}_u)\Big]^2B(V^{1,N}_u-V^{j,N}_u,\sigma)
\dd \sigma \dd u\Big]\\
=& \frac12\E\Big[\int_s^t \int_{\Sp_2} 
\Big[\phi(v'(V^{1,N}_u,V^{2,N}_u,\sigma))-\phi(V^{1,N}_u)\Big]^2B(V^{1,N}_u-V^{2,N}_u,\sigma)\dd \sigma \dd u\Big].
\end{align*}
Using that $\nabla \phi$ is bounded and~\eqref{est4}, we conclude that
$$
\E\Big[(M^{1,N}_{s,t})^2\Big]\leq C_\Phi \E\Big[\int_s^t |V^{1,N}_u-V^{2,N}_u|^{\gamma+2} \dd u\Big].
$$
This last quantity is bounded by~\eqref{tax2}, since $\gamma+2\in(0,2)$.
Thus $\E[A_N^2]$ tends to $0$, and $\E[B_N^2]$ is treated similarly.
We have proved~\eqref{jab1}.

\vip

Next, we introduce, for $A\geq 1$, $\Phi_A$ defined exactly as $\Phi$ in \eqref{def-Phi} but  
with $\bar \cA \phi(x(u),y(u))$ replaced by 
$\chi_A(\bar \cA \phi(x(u),y(u)))$, where $\chi_A(r)=(-A) \lor r \land A$.
The map $\Phi_A : \D(\rr_+,\rr^3\times\rr^3)$ is bounded and continuous at any point 
$(x,y)\in \D(\rr_+,\rr^3\times\rr^3)$ such that $x$ a.s. has no jump at times $s_1,\dots,s_k,s,t$.
Consequently, the map $\MM\mapsto |\langle \MM, \Phi_A\rangle|$
is continuous and bounded at any $\MM \in \cP(\D(\rr_+,\rr^3\times \rr^3))$
such that $\MM(\{(x,y) : \Delta x(s_1)=\dots=\Delta x(s_k)=\Delta x(s)=\Delta x(t)=0\})=1$.
By definition of $J$, $\mu\otimes \mu$ a.s. satisfies this condition, since  $s_1,\dots,s_k,s,t \in \rr_+\setminus J$.
Since $\mu^N \odot\mu^N$ converges in law to $\mu\otimes \mu$, we conclude 
that
\begin{equation}\label{jab2}
\E[|\langle\mu\otimes \mu, \Phi_A \rangle|] = \lim_N \E[|\langle\mu^N\odot \mu^N, \Phi_A \rangle|].
\end{equation}

\vip
We now check that
\begin{gather}
\lim_{A\to \infty}\sup_N \E[|\langle\mu^N\odot \mu^N, \Phi-\Phi_A \rangle|+
|\langle\mu\otimes \mu, \Phi-\Phi_A \rangle|]=0. \label{jab3}
\end{gather}
There exists a constant $C_\Phi$ such that for all $(x,y)\in  \D(\rr_+,\rr^3\times\rr^3)$,
$$
|\Phi(x,y)-\Phi_A(x,y)|\leq C_\Phi \int_s^t |\bar \cA\phi(x(u),y(u))|\indiq_{\{|\bar \cA\phi(x(u),y(u))|\geq A\}} \dd u.
$$
But $|\bar \cA\phi(v,v_*)|\leq C_\phi(|v-v*|^{\gamma+1}+|v-v*|^{\gamma+2})$ by~\eqref{est5p1}.
Hence for any $p>1$,
$$
|\Phi(x,y)-\Phi_A(x,y)|\leq \frac{C_\Phi'}{A^{p-1}}\int_s^t (|x(u)-y(u)|^{\gamma+1}+|x(u)-y(u)|^{\gamma+2})^p \dd u,
$$
and it only remains to show that we can choose $p>1$ such that 
$$
\sup_{N\geq 2} \E[\langle\mu^N\odot \mu^N,\Psi\rangle] + \E[\langle\mu\otimes \mu,\Psi\rangle]<\infty, 
$$
where
$$
\Psi(x,y)=\int_s^t (|x(u)-y(u)|^{\gamma+1}+|x(u)-y(u)|^{\gamma+2})^p \dd u.
$$
By definition of $\mu^N\odot \mu^N$ and by exchangeability, we have
$$
\E[\langle\mu^N\odot \mu^N,\Psi\rangle]= 
\E\Big[\int_s^t (|V^{1,N}_s-V^{2,N}_u|^{\gamma+1}+|V^{1,N}_s-V^{2,N}_u|^{\gamma+2})^p \dd u \Big],
$$
which is indeed bounded, thanks to~\eqref{tax} and~\eqref{tax2} if $p>1$ satisfies 
$-2 < p(\gamma+1)<2$ and $-2<p(\gamma+2)<2$. Since $\gamma \in (-2,0)$, such a $p$ exists.
One shows the finiteness of $\E[\langle\mu\otimes \mu,\Psi\rangle]$ similarly, using
Steps~2.1 and~2.3.

\vip
We finally conclude. We write, for any $N\geq 2$ and and $A\geq 1$,
\begin{align*}
\E[|\langle \mu \otimes \mu,\Phi\rangle|]\leq&
\E[|\langle \mu \otimes \mu,(\Phi-\Phi_A)\rangle|]\\
&+\Big|\E[|\langle \mu \otimes \mu,\Phi_A\rangle |]-\E[|\langle \mu^N \odot \mu^N,\Phi_A\rangle|]\Big| \\
&+ \E[|\langle \mu^N \odot \mu^N,\Phi_A-\Phi\rangle |] \\
&+  \E[|\langle \mu^N \odot \mu^N,\Phi\rangle |].
\end{align*}
Taking the $\limsup$ as $N\to \infty$, the second and last term disappear thanks to~\eqref{jab2} and~\eqref{jab1}.
Taking the $\lim$ as $A\to \infty$, we deduce from~\eqref{jab3}
that $\E[|\langle \mu \otimes \mu,\Phi\rangle|]=0$, which was our goal.
\vip

{\it Step 3.} We finally prove (iii). By (i), we know that the sequence $(\mu^N, N\geq 2)$ is tight in 
$\cP(\D(\rr_+,\rr^3))$. By (ii), we know that any limit point $\mu$ of this sequence a.s.
solves $MP(f_0)$ and satisfies $\sup_{t\geq 0} \E[I_1(\mu_t)]\leq I_1(f_0)$, which implies
that $\E[\int_0^T I_1(\mu_s)\dd s] \leq T I_1(f_0)$ for all $T>0$, so that 
$t\mapsto I_1(\mu_t) \in L^1_{loc}([0,\infty))$ a.s. Hence $\mu=\QQ$ a.s., where $\QQ$ is the unique solution
to $MP(f_0)$ such that $t\mapsto I_1(\QQ_t) \in L^1_{loc}([0,\infty))$, see Proposition~\ref{unimp}.
All this shows that $\mu^N$ goes in law to $\QQ$ as $N\to\infty$. 
The limit $\QQ$ being deterministic, the convergence also holds in probability.
\end{proof}

\section{Entropic propagation of chaos}
\label{sec:propagation}

Finally, we give the 

\begin{proof}[Proof of  Theorem~\ref{th:chaos}] 
We first show that $\QQ(\Delta V^*_t\neq 0)=0$ for all $t\geq 0$,
where we recall that $(V^*_t)_{t\geq 0}$ is the canonical process of $\D(\rr_+,\rr^3)$. By Theorem~\ref{mr2}
and the facts recalled in Subsection~\ref{subsec:propgation-chaos}, we know that
$((\bV^N_t)_{t\geq 0})_{N\geq 2}$ is $\QQ$-chaotic, so that
$(V^{1,N}_t)_{t\geq 0}$ goes in law to $\QQ$ as $N\to \infty$. But the tightness of the family
$((V^{1,N}_t)_{t\geq 0},N\geq 2)$ has been checked through the Aldous criterion, see the proof of Lemma~\ref{tight}.
This implies, see Jacod-Shiryaev~\cite[Section VI, Remark 4.7]{MR1943877}, 
that any limit point of $((V^{1,N}_t)_{t\geq 0},N\geq 2)$ is quasi-left 
continuous, so that in particular $\QQ(\Delta V^*_t\neq 0)=0$ for all $t\geq 0$.

\vip
For each $t\geq 0$, the map $\pi_t : \D(\rr_+,\rr^3) \to \rr^3$ defined by $\pi_t(x)=x(t)$ is continuous
at any point $x$ such that $\Delta x(t)=0$. Thus 
for each $t\geq 0$, the map $\Pi_t : \cP(\D(\rr_+,\rr^3))\to \cP(\rr^3)$ defined by 
$\Pi_t(q)=q\circ \pi_t^{-1}$ is continuous at any point $q$ such that $q(\Delta V^*_t\neq 0)=0$.

\vip

All this implies that for each $t\geq 0$, $(\bV^N_t)_{N\geq 2}$ is 
$\QQ_t$-chaotic.
But we know from Proposition~\ref{ev} that $\QQ_t=f_t$, where
$(f_t)_{t\geq 0}$ is the unique weak solution to~\eqref{be} introduced in Theorem~\ref{thm:exist&uniqBoltz}. 
It only remains to prove that for each $t\geq 0$, $H_N(F^N_t) \to H_1(f_t)$ as $N\to \infty$, 
but this follows from~\cite[Theorem~1.4]{MR3188710}, since $I_N(F^N_t) \le I_N(F^N_0) = I_1(f_0) < \infty$.
\end{proof} 

\bibliographystyle{acm}
\bibliography{bib-chaosBoltzmann}

\end{document}